\documentclass{article}
\usepackage{proof}
\usepackage{amssymb}
\usepackage{amsfonts}
\usepackage{latexsym}
\usepackage{epsfig}
\usepackage{xspace}
\usepackage[usenames]{xcolor}
\usepackage{amstext}
\usepackage{rotating}
\usepackage{amsmath,amssymb}
\usepackage{amsthm}
\usepackage{relsize}
\usepackage{enumerate}
\usepackage{dutchcal}
\newcommand{\verz}[1]{\{ #1 \}}
\newcommand{\tupel}[1]{\langle #1 \rangle}
\newcommand{\bident}[0]{\bigskip\noindent}
\newcommand{\qee} {\hspace*{2mm}\hfill $\shapipe$}
\newcommand{\shapipe}{\,
                    \setlength{\unitlength}{1ex}
                    \begin{picture}(2,2)
                    \put(.25,.15){\line(0,1){1.22}}
                    \put(.25,1.37){\line(1,0){0.61}}
                    \put(.45,-.05){\line(1,0){0.61}}
                    \put(.4,0){\line(1,0){0.61}}
                    \put(.35,.05){\line(1,0){0.61}}
                    \put(.3,.1){\line(1,0){0.61}}
                    \put(.25,.15){\line(1,0){0.61}}
                    \put(1.04,-.05){\line(0,1){1.22}}
                    \put(0.99,0){\line(0,1){1.22}}
                    \put(0.94,.05){\line(0,1){1.22}}
                    \put(0.89,.1){\line(0,1){1.22}}
                    \put(0.84,.15){\line(0,1){1.22}}
                    \end{picture}
                    \!}
\newcommand{\Iff}{\Leftrightarrow}

\newcommand{\To}{\Rightarrow}

\newcommand{\la}{\langle}
\newcommand{\ra}{\rangle}
\newcommand{\s}{{\sf s}}
\newcommand{\nm}[1]{#1}
\newcommand{\sma}{\mathbin{\#}}
\newcommand{\gnum}[1]{{\ulcorner #1 \urcorner}}
\newcommand{\bleq}{\mathbin{\leq}}
\newcommand{\bles}{\mathbin{<}}
\renewcommand{\phi}{\varphi}

\renewcommand{\Theta}{\varTheta}
\renewcommand{\Phi}{\varPhi}
\renewcommand{\Psi}{\varPsi}
\renewcommand{\Xi}{\varXi}
\renewcommand{\Omega}{\varOmega}
\renewcommand{\Gamma}{\varGamma}
\newcommand{\svo}{{\mathfrak a}}
\newcommand{\svt}{{\mathfrak b}}
\newcommand{\svtr}{{\mathfrak c}}

\newcommand{\inty}{interpretability\xspace}

\newcommand{\rat}{reasonable arithmetical theories\xspace}

\theoremstyle{plain}
\newtheorem{theorem}{Theorem}[section]
\newtheorem{lemma}[theorem]{Lemma}
\newtheorem{proposition}[theorem]{Proposition}
\newtheorem{fact}[theorem]{Fact}

\newtheorem{corollary}[theorem]{Corollary}

\theoremstyle{definition}
\newtheorem{definition}[theorem]{Definition}

\newtheorem{ques}[theorem]{Open Question}

\newcommand{\formal}[1]{\ensuremath{{\sf {#1}}\xspace}}

\newcommand{\axioms}[2]{{\ensuremath{{\formal{axioms}}_{#1}(#2)\xspace}}}
\newcommand{\bewijs}[2]{{\ensuremath{{\formal{proof}}_{#1}(#2)\xspace}}}
\newcommand{\principle}[1]{\formal{#1}}

\newcommand{\finprin}[2]{{\ensuremath{\sf{#1^\mathnormal{#2}}}}\xspace}
\newcommand{\prin}[1]{{\ensuremath{\sf{#1}}}\xspace}

\newcommand{\sonetwo}{{\ensuremath{{{\sf S}^1_2}}}\xspace}

\newcommand{\pa}{\ensuremath{{\sf PA}}\xspace}

\newcommand{\gl}{{\sf GL}\xspace}

\newcommand{\expo}{{\ensuremath{{\sf exp}}}\xspace}

\newcommand{\eqbydef}{:=}

\newcommand{\finboxext}[1]{\principle{\translated{(E\Box)}{#1}}\xspace}
\newcommand{\finintext}[1]{\principle{\translated{(E\rhd)}{#1}}\xspace}
\newcommand{\finimpl}[1]{\principle{\translated{(\rightarrow \Box)}{\mathnormal{#1}}}\xspace}

\newcommand{\fint}[3]{{\ensuremath{{#1}^{[{#2},{#3}]}}}\xspace}
\newcommand{\finmod}[2]{\,{\ensuremath{{#1}^{#2}}}\xspace}
\newcommand{\finrhdimpl}[1]{\principle{(\rightarrow\rhd)^\mathnormal{#1}}\xspace}

\newcommand{\il}{{\ensuremath{\textup{\textsf{IL}}}}\xspace}
\newcommand{\extil}[1]{\ensuremath{\sf IL#1}\xspace}
\newcommand{\ilm}{\extil{M}}
\newcommand{\ilp}{\extil{P}}
\newcommand{\ilpnaught}{\extil{P_0}}

\newcommand{\intl}[1]{{\ensuremath {\textup{\textbf{IL}}}({\rm #1})}}
\newcommand{\ilal}{\intl{{\sf All}}}

\newcommand{\translated}[2]{{\ensuremath{{#1}^{#2}}\xspace}}
\newcommand{\trans}[2]{{\ensuremath{{#1}^{#2}}\xspace}}

\newcommand{\pvu}{A}

\newcommand{\avu}{\varphi}
\newcommand{\avd}{\psi}
\newcommand{\avt}{\chi}
\newcommand{\avq}{\alpha}
\newcommand{\avqq}{\beta}
\newcommand{\avs}{\sigma}

\renewcommand{\qedsymbol}{$\dashv$}

\newcommand{\luka}{\textcolor{red}{\textbf{LUKA:}}}

\newcommand{\atl}{{\sf FIL}\xspace}
\newcommand{\appr}[2]{ \la{#1}, {#2}\ra }

\newcommand{\ignore}[1]{ }
\newcommand{\forlater}[1]{ }

\definecolor{uuxgreen}{cmyk}{1,0,0.75,0}
\definecolor{uuxred}{cmyk}{0.2,1,0.9,0.1}
\definecolor{uuyblue}  {cmyk}{0.9,0.55,0,0}
\definecolor{uuxblue}  {cmyk}{0.9,0.55,0,0}
\definecolor{grre}{rgb}{0.7,0.5,0.5}
\definecolor{grbl}{rgb}{0.4,0.5,0.7}
\definecolor{lightgray}{rgb}{0.75, 0.75, 0.75}
\newcommand{\albnote}[1]{\marginpar{\footnotesize \textcolor{uuyblue}{#1}}}
\newcommand{\luknote}[1]{\marginpar{\footnotesize \textcolor{uuxred}{#1}}}
\newcommand{\joonote}[1]{\marginpar{\footnotesize \textcolor{uuxgreen}{#1}}}
\newcommand{\bnb}[1]{#1}

\title{Feferman Interpretability} 
\author{Joost J. Joosten, Luka Mikec, and Albert Visser}

\begin{document}

\maketitle

\abstract{\noindent
We present an interpretability logic \atl or \emph{Feferman Interpretability Logic}. The $\Box$ provability modality 
can occur in \atl with a label, as in $\Box^\svo$. Likewise the $\rhd$ interpretability modality can occur in \atl with a label, as in $\rhd^\svo$. 
The labels indicate that in the arithmetical interpretation, the axiomatisation of the base theory will be tweaked/customised. 
\forlater{\joonote{I think we need another word for approximation. Intensionalising the axiom set?}\albnote{The axiom set is already intensional. Customising is a nice word. At some point we use tweaking which I like.}}
The base theory $T$ will always contain the minimum of \sonetwo  of arithmetic and $T$ will be 
approximated by $T^a$ in such a way that $T$ is extensionally the same as $T^a$. However, 
$T^a$ will inherit certain properties reminiscent of finitely axiomatised theories. 

After providing the logic \atl and proving the arithmetical soundness,
we set the logic to work to prove various interpretability principles to be sound in a large variety of 
(weak) arithmetical theories. In particular, we prove the two series of principles 
from \cite{GorisJoosten:2020:TwoSeries} to be arithmetically sound using \atl. 
Up to date, the arithmetical soundness of these series had only been proven using the techniques of definable cuts.

\section{Preludium}
Interpretability Logic is an approach to the study of interpretability. Unlike the study of
interpretability degrees and categories of theories and interpretations, the distinctive feature
of interpretability logic is the internalisation and nesting of interpretability viewed
as a modal connective. For example, interpretability logic allows us to study what
the \emph{internal} verification of the model existence lemma means in formal theories
(Principle ${\sf J}_5$; see Section~\ref{lekkerbeksmurf}).

In the case of classical theories, for the primary reading of the
modal connectives, there is a marked difference between provability logic and
interpretability logic. Where provability logic is remarkable stable: no arithmetical
theories with significantly different provability logics have been discovered,
substantially different interpretability logics are realised in different (classes of) theories.

Interpretability Logic turns out to be
a land of two streams. Its Euphrates is the logic {\sf ILM}
and its Tigris the logic {\sf ILP}. The logic {\sf ILM} is the logic of \emph{essentially reflexive sequential theories}, alternatively characterised as
sequential theories with full induction w.r.t. a designated interpretation of a theory of the natural numbers.\footnote{We think that
the class of theories realising {\sf ILM} can be extended to a class of \emph{essentially sententially reflexive sequential theories}. We probably
need our arithmetical base to be $\mathrm I\Sigma_1$. However, this possibility has not been studied.
See \cite{bekl:limi05} for some relevant results.} The theory consists of 
the base logic {\sf IL}, given in Section~\ref{lekkerbeksmurf} plus the principle {\sf M}:
$A\rhd B \to (A\wedge \Box C)\rhd (B\wedge\Box C)$. The logic {\sf ILP} is the interpretability logic of finitely axiomatised theories that interpret
${\sf EA}^+$, a.k.a. $\mathrm I\Delta_0+{\sf supexp}$. The logic is given by {\sf IL} plus the principle {\sf P}: $A\rhd B \to \Box(A\rhd B)$.

Both logics we introduced around 1987 by Albert Visser. A modal semantics for the theories was
discovered soon after by Frank Veltman. See, e.g., \cite{JonghVeltman:1990:ProvabilityLogicsForRelativeInterpretability} and \cite{JonghVeltman:1999:ILW}.
The arithmetical completeness of {\sf ILM} was proved by Alessandro Berarducci in \cite{Berarducci:1990:InterpretabilityLogicPA} and by Volodya Shavrukov
in \cite{Shavrukov:1988:InterpretabilityLogicPA}. The arithmetical completeness of {\sf ILP} was proved by Albert Visser in his paper \cite{Visser:1990:InterpretabilityLogic}.
For more information, see e.g. \cite{Visser:1997:OverviewIL}, \cite{Japaridze:1998:HandbookPaper}, \cite{arte:prov04}.

But what happens if we distance ourselves from the rivers? There is a scarcity of results for specific theories. 
We do have a Kripke model characterisation
of the interpretability logic of {\sf EA}, aka $\mathrm I\Delta_0+{\sf exp}$, however, we do not have an axiomatisation. See \cite{Kalsbeek:1991:TowardsExp}.
Another case is Primitive Recursive Arithmetic. This is a theory that is neither finitely axiomatisable 
nor essentially reflexive. Some modest results have been obtained towards its interpretability logic but the full characterisation is still open
(\cite{JoostenIcard:2012:RestrictedSubstitutions, BilkovaJonghJoosten:2009:PRA}).

The most salient question is: what is the interpretability logic of all reasonable theories? This is a koan-like question since \emph{what is reasonable?} is
part of the question. A preliminary study was done in \cite{JoostenVisser:2000:IntLogicAll}.
See also \cite{GorisJoosten:2020:TwoSeries} where a list of principles is given and verified.
The principles valid in all reasonable theories will certainly be in the intersection of {\sf ILM} and {\sf ILP}.
An example of such a principle is {\sf W}: $A \rhd B \to A \rhd (B\wedge \Box \neg A)$. This principle has both
an {\sf ILM}- and an {\sf ILP}-proof. Interestingly, we can generalise the {\sf ILM}-proof to a wide class of theories,
to wit sequential theories where the interpretation of the theory of numbers satisfies \sonetwo. The basic idea
here is that we can view the {\sf ILM}-proof as using the insight that, for all models $\mathcal M$ of our
sequential essentially reflexive theory $T$, any internal model is an end-extension of the $\mathcal M$-internal $T$-models.
This insight has a trace in all sequential theories (as discovered by Pavel Pudl\'ak), 
to wit that $\mathcal M$ and its internal model $\mathcal N$ share a definable cut (modulo internally definable isomorphism).

We can also generalise the {\sf ILP}-proof. To do that we use a trick due to Feferman (\cite{feferman1960arithmetization})
to make a theory behave as if it were finitely axiomatised by modifying the representation of the axiom set. 
The {\sf P}-style proof of {\sf W} has even wider scope: it holds
for all theories (with decent axiom sets) that interpret \sonetwo.
In analogy to {\sf W}, many other principles can be given {\sf M}-style proofs and {\sf P}-style proofs with wider scope.

The aim of the present paper is to systematically study the {\sf P}-style methodology and Feferman's trick.
We do this by developing a modal logic that is specifically built to implement this methodology.
Our present paper is a genuine extension of an earlier paper by Joost Joosten and Albert Visser, to wit \cite{JoostenVisser:2004:Toolkit}.

\forlater{
\medskip
\textcolor{uuxblue}{\cite{JonghVeltman:1990:ProvabilityLogicsForRelativeInterpretability}.
\cite{dejo:expl91},  \cite{arec:inte98}, \cite{JonghVeltman:1999:ILW}, \cite{Japaridze:1998:HandbookPaper}. 
\cite{JonghJumeletMontagna:1991}
}\albnote{Do check whether we want to slip more references in.}
}

\forlater{\section{Contents of the paper}

\textcolor{uuxblue}{The paper contains something for everyone. You will be go in a scenic tour through
fascinating landscapes of the mind; see things that can't be unremembered. This thrilling experience
will be with you your whole life.}

 This logic is a sublogic both of {\ilm} and
{\ilp}, but it is not the intersection of {\ilm} and {\ilp}.
(See \cite{Visser:1997:OverviewIL}.)
}

\section{Preliminaries}

In this section we revisit the basic definitions and results needed in the rest of the paper; 
definitions and results from arithmetic, formalised metamathematics and modal interpretability logics. 
NB: I only used the macros where the wrong variable was used and I did this only for Sections 1--4.}

\subsection{Arithmetic}

In this paper we will be using reasoning in and over weak arithmetics. 
To this end, let us start by describing the theory \sonetwo,
    introduced by Buss in \cite{Buss:1986:BoundedArithmetic}.
This is a finitely axiomatisable and weak first-order theory of arithmetic.
The signature of \sonetwo{} is \[
    (0,\, \s,\, | \cdot |,\, \lfloor \frac 1 2 \cdot \rfloor,\, +,\, \times,\, \sma,\, =,\, \leq).
\]
The intended interpretation of $|\cdot|$ is the length of its argument when expressed
    in the binary number system.
In other words, $|n|$ is (in the intended interpretation)
    equal to $\lceil \log_2 (n + 1) \rceil $.
The intended interpretation of $\lfloor \frac 1 2 \cdot \rfloor$
    is precisely the one suggested by the notation: dividing the argument by two and rounding the result downward. 
The symbol \# is pronounced `smash' and has the following intended interpretation
    (``the smash function''): \[
    n \sma m = 2^{ |n| |m| }.
\]
The remaining symbols are to be interpreted in the expected way.

The motivation for the smash function is that it gives an upper
    bound to Gödel numbers of formulas obtained by substitution:
Suppose $\avu$ is a formula, $x$ a variable and $t$ a term.
Given the Gödel numbers of $\avu$ and $t$ 
    (denoted with $\gnum \avu $ and $\gnum t $, as usual), 
    the Gödel number of $\avu(x \mapsto t)$ will not surpass 
    $\gnum \avu \sma \gnum t$. Of course, we need a `natural' G\"odel numbering to make this happen. See below.

Here and in the remainder of this paper, the assumption is that both the numeral representation and
    the Gödel numbers we work with are efficient.
For example, we can take the Gödel number of a string of symbols 
    to be its ordinal number in an arbitrary computationally very easy
    but otherwise fixed enumeration of all strings in the language of \sonetwo.
As for the numerals,         
    we 
    use \emph{efficient numerals}, defined recursively as follows: 
    \begin{align*}
        \underline{0} &\mapsto 0;\\
        \underline{2n + 1} &\mapsto \s(\s\s 0 \times \underline{n}); \\
        \underline{2n+2} &\mapsto \s\s(\s\s 0 \times \underline{n}).
    \end{align*}
Clearly, efficient numerals have about 
    the same growth rate as the corresponding binary representations.    
    
We also require that the code of a subterm is 
    always smaller than the entire term, and, similarly, for formulas.
We will consider such codings to be \emph{natural}.
See \cite{Buss1998} for details.

An example of such a natural coding is the Smullyan coding where
we code a string of letters (in a given alphabet of prime cardinality) as its number in the
length-first ordering.

Before introducing (some of) the axioms of \sonetwo,
    we will first define a certain hierarchy of formulas in the language of \sonetwo.
We will say that a quantifier is \emph{bounded}
    if it is of the form $(Q x \bleq t)$ 
    where $t$ is a term\footnote{
By ``$(Q x \bleq t)$'' we mean ``$(\exists x)(x \leq t \wedge \dots)$'',
    if $Q$ is $\exists$, and 
    $(\forall x)(x \leq t \to \dots)$'' if $Q$ is $\forall$.}
    that does not involve $x$.

A quantifier is \emph{sharply bounded} if it is of the form $(Q x \bleq |t|)$ 
    where $t$ is a term that does not involve $x$
    
\begin{definition}[\cite{Buss1998}]
    Let $\Delta_0^{\sf b}$, $\Sigma_0^{\sf b}$, and $\Pi_0^{\sf b}$ stand 
        for the set of formulas all of whose quantifiers are sharply bounded.
    We define $\Delta_i^{\sf b}$, $\Sigma_i^{\sf b}$, and $\Pi_i^{\sf b}$ for $i > 0$ 
        as the minimal sets satisfying the following conditions:
    \begin{enumerate}
        \item If $\avu$ and $\avd$ are $\Sigma_i^{\sf b}$-formulas, then 
            $(\avu \wedge \avd)$ and $(\avu \vee \avd)$ are $\Sigma_i^{\sf b}$-formulas.
        \item If $\avu$ is a $\Pi_i^{\sf b}$-formula and $\avd$ is a $\Sigma_i^{\sf b}$-formula, 
            then $\neg\, \avu$ and $(\avu \to \avd)$ are $\Sigma_i^{\sf b}$-formulas.
        \item If $\avu$ is a $\Pi_{i-1}^{\sf b}$-formula, then $\avu$ is a $\Sigma_i^{\sf b}$-formula.
        \item If $\avu$ is a $\Sigma_i^{\sf b}$-formula, $x$ a variable and $t$ is a term 
            not involving $x$, then $(\forall x \bleq |t|)\,\avu$ is a $\Sigma_i^{\sf b}$-formula.
        \item If $\avu$ is a $\Sigma_i^{\sf b}$-formula, $x$ a variable and $t$ is a term 
            not involving $x$, then $(\exists x \bleq t)\,\avu$ and $(\exists x \bleq |t|)\,\avu$ 
            are $\Sigma_i^{\sf b}$-formulas. 
        \item The first five conditions are to be repeated in the dual form:
            with the roles of $\Sigma$ and $\Pi$, and $\exists$ and $\forall$, swapped in all places.
        \item A formula $\avu$ is a $\Delta_i^{\sf b}$-formula if it is equivalent over predicate logic both to a
            $\Sigma_i^{\sf b}$-formula and to a $\Pi_i^{\sf b}$-formula.
    \end{enumerate}
\end{definition}   
Thus, this hierarchy is analogous to the standard arithmetical hierarchy,
    with bounded quantifiers in the role of unbounded quantifiers,
    and sharply bounded quantifiers in the role of bounded quantifiers.

\begin{definition}[The polynomial induction schema \cite{Buss1998}]
    Let $\Phi$ be a set of formulas which may contain zero or more free variables. 
    We define $\Phi$-PIND axioms to be the formulas \[
        \avu(x := 0) \wedge (\forall x)\, \big(\avu(x := \lfloor \frac 1 2 x \rfloor) \to \avu \big) \to (\forall x) \avu,
    \] 
    for all $\avu \in \Phi$ and all variables $x$.
\end{definition}
Thus, when proving facts using the schema of polynomial induction, 
    in the inductive step we are allowed to refer to the property
    obtained for $\lfloor \frac 1 2 n \rfloor$.
This is, of course, faster than the standard schema of mathematical induction where we     
    can use the property obtained for $n - 1$. The price we pay is a stronger antecedent in the induction principle.

We obtain \sonetwo by extending a certain list of 32 quantifier-free formulas 
    (dubbed \textsf{BASIC}, see e.g.\ \cite{Buss1998})
    with all $\Sigma_1^{\sf b}$-PIND axioms.
    
This somewhat unusually axiomatised theory has a nice connection
    to computational complexity, as the next theorem shows.
\begin{theorem}[\cite{Buss:1986:BoundedArithmetic}]
    \label{thm:buss-polytime}
    We have the following.
    \begin{itemize}
        \item     Suppose $\sonetwo \vdash (\forall x)(\exists y)\,\avu(x, y)$ for some $\Sigma_1^{\sf b}$-formula $\avu$.
            Then there is a \textsf{PTIME}-computable function $f_\avu$ such that
            if $f_\avu(x) = y$ then $\avu(x, y)$ holds \textup($f_\avu$ is a \emph{witnessing function} for $\avu$\textup),
            and  $\sonetwo \vdash (\forall x)\,\avu(x, f_\avu(x))$.
        \item Conversely, suppose $f$ is a \textsf{PTIME}-computable function.
            Then there is a $\Sigma_1^{\sf b}$-formula $\avu_f$ such that 
                $\avu_f(x, y)$ holds if and only if $f(x) = y$, and
                $\sonetwo \vdash (\forall x)(\exists y)\,\avu_f(x, y)$.
    \end{itemize}
\end{theorem}    
Theories in this paper will be $\Delta^{\sf b}_1$-axiomatised
    theories (i.e.\ having \textsf{PTIME}-decidable axiomatisations). Moreover, we will always assume that any theory we consider comes with a designated interpretation of ${\sf S}^1_2$. 

That is, when we say ``a theory'', we mean a pair of an actual theory
    together with some singled-out and fixed interpretation of ${\sf S}^1_2$. 

    
A principle similar to induction is that of \emph{collection},
    in particular $\Sigma_1$-collection.
\begin{definition}[$\Sigma_1$-collection]
    The schema \[
        (\forall n)\Big((\forall x \bles n)(\exists y) \avu(x, y) \to
        (\exists m)(\forall x \bles n)(\exists y \bles m) \avu(x, y)\Big)
    \]
    where $\avu$ is restricted to $\Sigma_1$-formulas possibly with parameters, is the $\Sigma_1$-collection schema.
\end{definition}    
Collection is occasionally useful, however we will have to find 
    ways to avoid it as it is not available in \sonetwo.

    \subsection{Interpretability}\label{voetbalsmurf}

    We refer the reader to \cite{Visser:1991:FormalizationOfInterpretability} or \cite{visser:2018:InterpretationExistence} 
    for the definitions of translation and interpretation.
    
    There is one point specific to this paper. We want to treat a translation $k$ as an interpretation, in a given theory $T$, of an unspecified target theory
    in a given target signature $\Theta$.
    To fulfill this role, $T$ needs to prove at least the $k$-translations of the axioms of identity for signature $\Theta$.
    However, generally, this may fail.
     The reason is that, even if identity as a logical connective, we treat it in translation simply as
    a symbol from the signature. In other words, we translate identity not necessarily to identity. Also, we need
    the guarantee that the domain is non-empty to satisfy the axiom $\exists x\, x=x$.
    In fact, the usual treatment of interpretations fits free logic without identity best.
    
    We  consider only finite signatures, so the theory
    of identity for signature $\Theta$ will be given by a single axiom, say $\mathfrak{id}_\Theta$. 
    Thus, what we need for the translation $k$ to carry an interpretation at all is that $T \vdash \mathfrak{id}^k_\Theta$. 

    We implement a simple hack to ensure that every translation carries an interpretation to some theory. 
    We fix a default translation $m$ that interprets $\mathfrak{id}_\Theta$ in $T$. We can take as the domain of $m$
     the full domain of $T$ and translate identity to identity. The translation of the predicate symbols can be
    arbitrarily chosen. We can now replace $k$ by the disjunctive interpretation $k^\ast$ that is $k$ in case $\mathfrak{id}^k_\Theta$
    and $m$ otherwise. Clearly, we will always have $T\vdash \mathfrak{id}^{k^\ast}_\Theta$. Moreover, if
    $T \vdash \mathfrak{id}^k_\Theta$, then $k$ and $k^\ast$ coincide modulo $T$-provable equivalence.
    The idea is now simply that the translation we quantify over are really the $k^\ast$, so that they always carry
    some interpretation.

    We note that, in the context of interpretability logics, we are interested in translations from a signature $\Theta$ to itself.
    In that context, we can take as the default translation $m$ simply the identity translation on $\Theta$.

    In Section~\ref{brutalesmurf}, we will strengthen our demand on translations somewhat to ensure that we do have coding in all theories
    we consider.

\subsection{Formalised provability and interpretability}

Before introducing formalised interpretability, 
    let us say a few words on formalised provability.
    For a given signature, we fix a natural formalisation ${\sf aproof}(p,x)$ of proof from assumptions.
    We usually leave the signature implicit. We assume that a proof from assumptions is given, Hilbert-style,
    as a sequence of pairs of a number giving the status of the inference step and a formula.\footnote{Of course,
    we do not really need the Hilbert format. However, the definition would be somewhat more
    complicated for, say, Natural Deduction.} Say $\mathfrak a$ tells us that the formula is
    an assumption. We can make {\sf aproof} a $\Delta_1^{\sf b}$-predicate.
    A theory $T$ comes equipped with a representation $\alpha$ of its axiom set. We will write ${\sf axioms}_T$ for $\alpha$.
    The default is that $\alpha$ is $\Delta_1^{\sf b}$. We write:
    \begin{itemize}
        \item 
        ${\sf proof}_T(p,x)$ for:\\
        \hspace*{1cm}${\sf aproof}(p,x) \wedge (\forall i \bles {\sf length}(p))\, ((p)_{i0} = \mathfrak a \to {\sf axioms}_T((p)_{i1}))$.
        \item 
        ${\sf Pr}_T(x)$ for $\exists p\; {\sf proof}_T(p,x)$
    \end{itemize}
    
   \noindent We note that, if $\alpha$ is $\Delta_1^{\sf b}$, then so is ${\sf prf}_T(p,x)$.

Let us denote the efficient numeral of the (natural) Gödel number of $A$
    by $\gnum A$.
Sufficiently strong theories (such as \sonetwo 
)
    prove the \textit{Hilbert--Bernays--Löb derivability conditions} (\cite{hilbert2013grundlagen}):
\begin{enumerate}
    \item for all $\avu$, if $T \vdash \avu$, then $T \vdash \mathsf{Pr}_T(\gnum \avu)$;
    \item for all $\avu$, $T \vdash \mathsf{Pr}_T(\gnum {\avu \to \avd}) \to (\mathsf{Pr}_T(\gnum \avu) \to \mathsf{Pr}_T(\gnum \avd))$;
        \item for all $\avu$,
        $T \vdash \mathsf{Pr}_T(\gnum \avu) \to \mathsf{Pr}_T(\gnum{\mathsf{Pr}_T(\gnum \avu)})$.
\end{enumerate}
These conditions, in combination with the Fixed Point Lemma, suffice to show that $T \vdash \mathsf{Pr}_T(\gnum {0=1})$
    and, consequently, $T \vdash 0 = 1$
    follows from $T \vdash \neg\, \mathsf{Pr}_T(\gnum {0=1})$,
    i.e.\ Gödel's second incompleteness theorem.
These conditions also suffice to show that the following holds: \[
    \text{if}\ T \vdash  \mathsf{Pr}_T(\gnum \avu) \to \avu, \ \text{then}\ T \vdash  \avu.
\]
Thus $T$ is only ``aware'' that $\mathsf{Pr}_T(\gnum \avu)$ implies $\avu$
    in case the conditional is trivially satisfied by the provability
    of its consequent.
This entailment is known as \emph{Löb's rule}.
In fact, $T$ is ``aware'' of this limitation (\emph{formalised Löb's rule}):
\[
    T \vdash \mathsf{Pr}_T(\gnum{\mathsf{Pr}_T(\gnum \avu) \to \avu}) \to \mathsf{Pr}_T(\gnum \avu).
\]
We can read e.g. the formulised L\"ob's rule as a propositional scheme by replacing  $\mathsf{Pr}_T$ with $\Box$ and
the variable $\avu$ that ranges over $T$-formulas by the variable $\pvu$ that rangesover propositional modal formulas.
The provability logic \gl is the extension of the basic modal logic \textsf{K}
    with an additional axiom schema representing Löb's formalised rule: \[
        \Box(\Box A \to A) \to A.
    \]
In his well-known result, 
    Solovay \cite{Solovay:1976} established arithmetical completeness for this logic. Upon inspection, this result
    works for all
    c.e. extensions of {\sf EA}, a.k.a. $\mathrm I\Delta_0+{\sf exp}$, that are $\Sigma_1$-sound.\footnote{In a wide
    range of cases, we can, given the theory, redefine the representation of the axiom set in such a way that one can drop the
    demand of $\Sigma_1$-soundness. See, e.g., \cite{visser:2021:absorption}.}

The predicate $\mathsf{Pr}_T$ satisfies the following property, which is known as the Kreisel Condition,
for $\exists\Sigma^{\sf b}_1$-sound theories: \begin{equation}
    \label{eq:provpred}
    T \vdash \avu \text{ if and only if }    \mathbb N \models \mathsf{Pr}_T(\gnum \avu).
\end{equation}
    
We can find alternative axiomatisations of $\mathsf{Pr}_T$,
    that satisfy Property \eqref{eq:provpred}, but behave differently w.r.t. consistency. 
One such axiomatisation is given in \cite{feferman1960arithmetization}.
Say the original axiomatisation is $\alpha$. We write $\alpha_x(y)$ for $\alpha(y) \wedge y\leq x$.
Let the theory axiomatised by $\alpha_x$ be $T_x$.
    We take: $\digamma(x)$ iff $\alpha(x) \wedge {\sf Con}(T_x)$.

    We note that we diverge from our default here: $\digamma$ is $\Pi_1$.
    We take $T^{\mathsmaller \digamma}$ to be the theory axiomatised by $\digamma$.
    We need that the theory $T$ is
    $\Sigma_1$-sound and reflexive to make \eqref{eq:provpred} work for ${\sf Pr}_{T^{\digamma}}$.
    
Let us call this notion \emph{Feferman-provability}.
As we are interested only in consistent theories,
    clearly this predicate has the same extension as the predicate $\mathsf{Pr}_T$.   
However, it is provable within \pa that $0 = 1$
    is not Feferman-provable. 
This is, of course, not the case with $\mathsf{Pr}_{\mathsmaller\pa}$,
    as that would contradict Gödel's second incompleteness theorem.    
    
If we are dealing with a theory $T$ with a poly-time decidable axiom set,
    by Theorem \ref{thm:buss-polytime}, there is a $\Sigma^{\sf b}_1$-predicate (actually $\Delta^{\sf b}_1$)
    verifying whether a number codes a $T$-proof of a formula.
This implies that the provability predicate, claiming that a proof exists
    for some given formula, is a $\exists \Sigma^{\sf b}_1$-predicate.
This is convenient because for \sonetwo{} we have provable 
     $\exists \Sigma^{\sf b}_1$-completeness: 
     \begin{theorem}
         For any $\exists \Sigma^{\sf b}_1$-formula $\avs$ we have 
         \[
         \sonetwo{} \vdash \avs \to \Box_T \avs.
         \]
     \end{theorem}
    
We now move on and consider interpretability.
There are various notions of formalised interpretability\footnote{see Theorem 1.2.10.\ of \cite{Joosten:2004:InterpretabilityFormalized,Joosten:2016:CharacterizationsInty}
    for a discussion on their relationships.}
Here we are interested in \emph{theorems-interpretability},
    i.o.w.~we say that $k$ is an interpretation of $V$ in $U$
    (we write $k:U\rhd V$) if and only if,
    $(\forall\avu)\, (\Box_V\avu\to \Box_U\avu^k)$.
Here $\Box_V$ and $\Box_U$ are 
    the provability predicates of $V$ and $U$, respectively.
We remind the reader that  theorems-interpretability is \sonetwo-provably transitive
---unlike axioms-interpretability.

The $k$-translation of $\avu$ is denoted as $\avu^k$.
If $V$ is a finitely axiomatisable theory,
    then $U\rhd V$ is in fact 
    a $\exists\Sigma^{\sf b}_1$ sentence.
This is due to the fact that, for finitely axiomatised theories $V$,
    their interpretability in $U$ boils down to the provability of the translation of the conjunction of their
    axioms and the fact that the translation function is {\sf P-TIME}.
As the theories studied in this paper are all $\Delta^{\sf b}_1$-axiomatisable,
    the aforementioned statement is $\exists \Delta^{\sf b}_1$,
    in particular $\exists \Sigma^{\sf b}_1$.

\subsection{Modal interpretability logics}\label{lekkerbeksmurf}
There are many different interpretability logics in the literature. 
The language of interpretability logics is that of propositional logic together with a unary modal operator $\Box$ and a binary modal operator $\rhd$. 
We adhere to the binding convention that the Boolean operators bind as usual and
that $\Box$ binds as strong as $\neg$ with all Boolean operators except $\to$ binding stronger than $\rhd$ and $\rhd$ binding stronger than $\to$. Thus, for example, 
$A\rhd B \to \Diamond A \wedge \Box C \rhd B \wedge \Box C$ will be short for:  
\[(A\rhd B) \to \Big( (\Diamond A \wedge \Box C) \rhd (B \wedge \Box C) \Big).\]
Most interpretability logics extend a core logic called \il.

\begin{definition} The logic \il has as axioms all tautologies in the modal propositional language containing $\Box$ and $\rhd$ together with all instances of 
the following axioms.
\[
\begin{array}{ll}
\prin{L_1} & \vdash \Box (A \rightarrow B) 
\rightarrow (\Box A \rightarrow {\Box} B) \\
\prin{L_2} &  \vdash{\Box} A \rightarrow {\Box}
{\Box} A  \\
\prin{L_3} & \vdash {\Box}
({\Box} A \rightarrow A) \rightarrow 
{\Box} A  \\\\
\prin{J_1} &  \vdash
\Box (A \rightarrow B) \rightarrow A \rhd B \\
\prin{J_2} &  \vdash (A \rhd B ) \wedge (B \rhd C)
\rightarrow A \rhd C \\
 %
\prin{J_3} &  \vdash (A \rhd C) \wedge (B \rhd C)
\rightarrow  A\vee B\, \rhd C \\
\prin{J_4} &  \vdash
 A \rhd B \rightarrow (\Diamond A \rightarrow \Diamond B)  \\
%
%
\prin{J_5} &  \vdash
\Diamond A \rhd A 
\end{array}
\]
The only rules are Necessitation $\prin{Nec}:\; {\vdash A} \ \To \ {\vdash \Box A}$ and Modus ponens. 
\end{definition}

We will consider extensions of \il by adding axiom schemes to \il. These logics will be named by appending the names of the new schemes to \il.
For example, the principle \principle{P} is given by $\vdash A\rhd B \to \Box (A\rhd B)$ and the logic \ilp arises by adding this scheme/principle to \il.
Likewise, the principle \principle{P_0} is given by $\vdash A\rhd \Diamond B \to \Box (A\rhd B)$ and the logic \ilpnaught arises by adding this scheme/principle to \il.

For later use, we prove the following easy observation.
\begin{lemma}\label{lemma:equivalentOfJ5}
If we replace in \il the axiom schema  
$\principle{J_5}:\;  \vdash\Diamond A \rhd A$ by $\principle{J_5}':\; \vdash B\rhd \Diamond C \to B\rhd  C$, then the resulting logic will be equivalent to the original logic \il.
\end{lemma}

\begin{proof}
Any formula $\Diamond A \rhd A$ is obtained from $B\rhd \Diamond C \to B\rhd  C$ by instantiating in the latter formula $B$ by $\Diamond A$ and $C$ by $A$. Thus, we get $\Diamond A \rhd \Diamond A \to \Diamond A\rhd  A$ since the antecedent is clearly provable without using $\principle{J_5}'$.

For the other direction, we reason in \il and assume $B\rhd \Diamond C$. Now, by \prin{J_5} we get $\Diamond C \rhd C$ so that by the transitivity axiom \principle{J_2} we obtain the required $B\rhd C$.
\end{proof}

\section{Tweaking the axiom set}\label{apthe}

For finitely axiomatised theories $V$, we 
have:
\[
    \sonetwo \vdash U \rhd V \rightarrow \Box_{\sonetwo} (U \rhd V),
\]
by $\exists \Sigma^{\sf b}_1$ completeness because $U \rhd V$ is  a $\exists\Sigma^{\sf b}_1$-sentence.
Recall that, in this paper, as a default, all theories are assumed to 
    be $\Delta^{\sf b}_1$-axiomatised.
If this were not the case, $U \rhd V$ need not, of course,
    be a $\exists\Sigma_1^{\sf b}$-sentence, even for finitely axiomatised theories $V$.
To mimic the \principle{P}-style behaviour for an arbitrary 
theory $V$, we will modify $V$\/ to a new theory $V'$\/ that approximates $V$\/
to obtain 
$\sonetwo \vdash U \rhd V \to\Box_{\sonetwo} (U \rhd V')$.
Of course, the new theory $V'$\/ should be sufficiently like  $V$\/
to be useful.
Thus, we  define a theory $V'$ that is extensionally
the same as $V$, but for which 
$U \rhd V'$ is a statement that is so simple that under the assumption that $U\rhd V$, we
can easily infer $\Box_{\sonetwo} (U \rhd V')$.

\subsection{The approximating theory defined}

We start with a first approximation.
Given some translation $k$, let us {\em define}\/ the set of  axioms  $V'$ as consisting of just 
those axioms $\phi$ of $V$\/ such that
 $U\vdash \phi^k$.
 Note that, if $k:U\rhd V$, then $V$ and $V'$ have the same axioms.
 However, when $V$ is not finitely axiomatisable, in general, we cannot take the insight $V\equiv V'$ with us when we proceed to
 reason inside a box. In formulas: we do have $k:U\rhd V \ \Rightarrow \ V \equiv V'$ but in general we do not have $k:U\rhd V \ \Rightarrow \ \Box ( V \equiv V')$.

  Notwithstanding, defining $V'$ as above is useful and works modulo some trifling details. Firstly, the definition of the
  new axiom set does not have the right complexity. Secondly, if the
  argument is not set up in a careful way, we may seem to need 
  both $\Sigma_1$-collection and \expo.
  We shall use a 
variation of Craig's trick so that the axiom sets that we consider 
will remain to be $ \Delta_1^{\sf b}$-definable.
The same trick makes the use of strong principles, like 
 $\Sigma_1$-collection and \expo,
superfluous.
   
\begin{definition}\label{defi:finapr}
Let $U$ and $V$ be  $ \Delta_1^{\sf b}$-axiomatised theories. Moreover, let $k$\/ be a translation of the language of $V$ into the language of $U$ that includes a domain specifier.

We remind the reader of Smory\'nski's dot notation. e.g., $\gnum{\dot p = \dot p}$ functions as a term that is the arithmetisation of the map
$p \mapsto \underline p=\underline p$.
Here is our definition of \fint{V}{U}{k}.
\begin{eqnarray*}
\axioms{\fint{V}{U}{k}}{x} & \eqbydef &
  (\exists\, p, \avu \bles x)\ 
\Big(x= {\sf conj}(\avu, \gnum{\dot p=\dot p})
\wedge \\ &&
  \hspace*{1.3cm}  \axioms{V}{\avu}\wedge \bewijs{U}{p,\trans{\avu}{k}} \Big).
\end{eqnarray*}

\end{definition}
\noindent
We note that this is
    a $\Sigma^{\sf b}_1$-formula. We can see that is equivalent to a $\Delta_1^{\sf b}$-formula by describing a procedure
    for deciding whether $\avd$ is a $\fint{V}{U}{k}$-axiom. 
    \begin{enumerate}[{\sf step} 1.]
        \item Is $\avd$ a conjunction? If not, $\avd$ does not qualify. Otherwise, proceed to the next step.
         \item Is the first conjunct of $\avd$, say $\avt$, a $V$-axiom? If not, $\avd$ does not qualify. Otherwise, proceed to the next step.
         \item 
         Is the second conjunct of the form $\underline p = \underline p $ and do we have \bewijs{U}{p,\trans{\avt}{k}}?
         If not, $\avd$ does not qualify. If so, $\avd$ will indeed be a $\fint{V}{U}{k}$-axiom.
    \end{enumerate}


The following lemma tells us that \sonetwo verifies that $k:U\rhd V$\/ implies that
 $V$\/ and $\fint{V}{U}{k}$ are {\em extensionally equal}. Actually, $V\rhd \fint{V}{U}{k}$ always holds and does not depend on the assumption $k:U\rhd V$.

\begin{lemma}\label{lemm:boxextensio}
Let $U$ and $V$ be  $ \Delta_1^{\sf b}$-axiomatised theories. We have
 \begin{enumerate}
 \item\label{bommel}
 $\sonetwo\vdash (\forall k)\,({\sf id}: V\rhd \fint{V}{U}{k})$.
 \item\label{tompoes}
 $\sonetwo \vdash (\forall k)\,( k:U\rhd V \to {\sf id}:\fint{V}{U}{k}\rhd V)$.
 \end{enumerate}
\end{lemma}

\begin{proof}
Ad (\ref{bommel}).  Reason in \sonetwo. We have to show:
$\Box_{\fint{V}{U}{k}}\varphi \rightarrow \Box_V \varphi$.
This is easily seen to be true, since we can replace every axiom 
$\varphi \wedge (\underline{p} = \underline{p})$ of \fint{V}{U}{k} by 
a proof of $\varphi \wedge (\underline{p} = \underline{p})$
from the  $V$-axiom $\varphi$. The resulting transformation
is clearly p-time.

\bident
Ad (\ref{tompoes}). Reason in \sonetwo. Suppose $k:U\rhd V$ and
$\Box_V\varphi$. We set out to prove $\Box_{\fint{V}{U}{k}}\varphi$. Let $p$ be a proof of $\varphi$ from $V$-axioms 
$\tau_0, \ldots ,\tau_n$. (Note that $n$ need not be standard.) 
We would be done, if we could replace every axiom occurrence of $\tau_i$ in $p$ by
$$
\infer[\wedge {\sf E},l]{\tau_i}{\tau_i \wedge 
(\underline{q_i}=\underline{q_i})}
$$
where $q_i$ would be a $U$-proof of $\tau_i^k$, so that we would obtain a 
$\fint{V}{U}{k}$-proof $r$\/ of $\varphi$. Clearly, for each $\tau_i$ we have that $\Box_V \tau_i$, so that by 
our assumption $k:U\rhd V$ we indeed obtain a $U$ proof $q_i$ of $\tau_i^k$. However, these proofs $q_i$ may be cofinal and, thus,
we would need a form of collection to exclude that possibility to keep the resulting syntactical object $r$ finite. 

It turns out that we can perform a little trick to avoid the use of collection. To this end, let $\tau$\/ be the (possibly non-standard) conjunction of these axioms $\tau_i$.
Note  that, by the naturality conditions on our coding, $\tau$ is bounded by $p$. Since, clearly, we have $\Box_{V}\tau$
we may find, using $k:U\rhd V$, a $U$-proof $q$\/ of
$\tau^k$. Here it is essential that we employ \textit{theorems interpretability} in this paper! 
We may use $q$
to obtain $U$-proofs of $q_i$ of 
$\trans{\tau_i}{k}$. Clearly,  we can extract appropriate $q_i$ from $q$ in such a way that $|q_i|$ is bounded by a term of order $|q|^2$.
We can now follow our original pland and replace every axiom occurrence of $\tau_i$ in $p$ by
$$
\infer[\wedge {\sf E},l]{\tau_i}{\tau_i \wedge 
(\underline{q_i}=\underline{q_i})}
$$
and obtain a 
$\fint{V}{U}{k}$-proof $r$\/ of $\varphi$.
We find that $|r|$ is bounded by a term of order $|p|\cdot|q|^2$. 
So, $r$ can indeed be found in p-time from the  given $p$ and $q$.
\end{proof}

For the previous lemma to hold it is essential that
 we work with efficient numerals  $\underline{p}$.
The reader may find it instructive to rephrase the lemma in terms of provability.

\begin{corollary}\label{cor:boxextensioProvability} 
For $U$ and $V$ being $\Delta_1^{\sf b}$-axiomatised theories we have 
 \begin{enumerate}
 \item\label{tommel}
 $\sonetwo\vdash (\forall k)\,(\forall \varphi)\, ( \Box_{\fint{V}{U}{k}} \varphi \to \Box_V \varphi)$;
 \item\label{bompoes}
 $\sonetwo \vdash (\forall k)\, \big( k:U\rhd V \to \forall \varphi\; ( \Box_{\fint{V}{U}{k}} \varphi \leftrightarrow \Box_V \varphi) \big)$.
 \end{enumerate}
\end{corollary}


\noindent
As mentioned before, even though we have extensional equivalence of $V$ and $\fint{V}{U}{k}$ under the assumption that $k:U\rhd V$, we do not necessarily have this under a provability predicate.
That is, although we do have
$\Box_{\sonetwo} (\Box_{\fint{V}{U}{k}}\varphi \rightarrow \Box_V \varphi)$ we
shall, in general, not have $k:U \rhd V \to \Box_{\sonetwo} (\Box_V \varphi \rightarrow 
\Box_{\fint{V}{U}{k}} \varphi)$.\\ 
\medskip

\subsection{A \principle{P}-like principle for the approximated theory}

The theory $\fint{V}{U}{k}$ is defined precisely so that it being interpretable in $U$ is true almost by definition. This is even independent on $k$ being 
or not an interpretation of $V$ in $U$. The following lemma reflects this insight.

\begin{lemma}\label{lemm:bijnap}
For $U$ and $V$,  $ \Delta_1^{\sf b}$-axiomatised theories we have 
 \[\sonetwo\vdash (\forall k)\,( k: U\rhd \fint{V}{U}{k}).\]
\end{lemma}

\begin{proof}
Reason in \sonetwo. Suppose $p$\/ is a $\fint{V}{U}{k}$-proof of $\phi$. 
 We want to construct a $U$-proof of $\phi^k$. As a first step we
 transform $p$\/ into a $V$-proof $p'$\/of $\phi$ as we did in the proof of
 Lemma~\ref{lemm:boxextensio},(\ref{bommel}): replacing all axioms $\varphi \wedge (\underline{s} = \underline{s})$ of \fint{V}{U}{k} by 
a proof of $\varphi \wedge (\underline{s} = \underline{s})$
from the  $V$-axiom $\varphi$. 

Next we transform 
 $p'$, using $k$, into a predicate logical proof $q$\/ of $\phi^k$\/ from assumptions 
 $\tau_i^k$, where each $\tau_i$\/ is a $V$-axiom. It is well known that this transformation
 is p-time. 
 
 Finally, each axiom $\tau_i$\/ extracted from $p$, comes from
 a $\fint{V}{U}{k}$-axiom $\tau_i\wedge (\underline r_i=\underline r_i)$, where 
 $r_i$\/ is a $U$-proof of $\tau_i^k$. So our final step is to extend $q$\/
 to a $U$-proof $q'$\/ by prepending the $U$-proofs $r_i$\/ above the corresponding $\tau_i^k$.
 This extension will at most double the number of symbols of $q$, so
 $q'\approx q^2$. 
\end{proof}

As a direct consequence of this lemma, we see via necessitation that $\sonetwo \vdash \Box_\sonetwo (\forall k)\,( k: U\rhd \fint{V}{U}{k})$ 
so that in a trivial way we obtain something that comes quite close to the \principle{P}-schema:
\begin{equation}\label{equation:AlmostPButNotYet}
\sonetwo \vdash U\rhd V \to  \Box_\sonetwo (\forall k)\,( k: U\rhd \fint{V}{U}{k}).
\end{equation}
However, Equation~\ref{equation:AlmostPButNotYet}, is somewhat strange, since the antecedent of the implication
does no work at all. In this paper, we are interested in finite extensions. Fortunately, a minor modification of  
Equation~\ref{equation:AlmostPButNotYet} does give information about finite extensions.

\begin{theorem}\label{theorem:PLikePrincipleAtL:general}
Let $U$ and $V$ be $\Delta_1^{\sf b}$-axiomatised theories.
We have:
\[\sonetwo \vdash (\forall k)\,(\forall \avu)\;  k: U\rhd (V+\avu) \to  \Box_\sonetwo\,  (k: U\rhd (\fint{V}{U}{k}+\avu)).\]
\end{theorem}

\begin{proof}
    We reason in \sonetwo. Suppose $U \rhd (V+A)$. It follows that $\Box_UA^k$, and, hence, $\Box_U\Box_UA^k$.
    
    We reason inside the $\Box_U$. We have both $\Box_U\avu^k$ and $k:U\rhd \fint{V}{U}{k}$. 
    We prove $U \rhd (\fint{V}{U}{k}+\avu)$.
    Consider any $V$-sentence $B$ and
    suppose $\Box_{\fint{V}{U}{k}+\avu}\avd$. It follows that $\Box_{\fint{V}{U}{k}}(\avu\to \avd)$. Hence, $\Box_U (\avu\to \avd)^k$.
    We may conclude that $\Box_U \avd^k$, so we are done.
\end{proof}

We will need the following thinned version of Theorem~\ref{theorem:PLikePrincipleAtL:general}, which shall 
be the final version of our approximation of the principle \principle{P}.

\begin{theorem}\label{theorem:PLikePrincipleAtL}
Let $T$ be a  $ \Delta_1^{\sf b}$-axiomatised theory 
and let $\avt$ and $\avq$ be $T$-sentences.
We have:
\begin{enumerate}[a.]
    \item 
    $\sonetwo \vdash (\forall k)\,( k: (T + \avq ) \rhd (T + \avqq) \to \Box_{\sonetwo}\, k: (T + \avq ) \rhd (\fint{T}{T+\avq}{k} + \avqq))$,
    \item 
    $\sonetwo \vdash (\forall k)\,( k: (T + \avq ) \rhd (T + \avqq) \to \Box_{\sonetwo}\, k: (T + \avq ) \rhd (\fint{T}{T}{k} + \avqq))$.
\end{enumerate}
\end{theorem} 

\begin{proof}
For (a), we apply Theorem~\ref{theorem:PLikePrincipleAtL:general} to $T+\avq$ in the role of $U$, $T$ in the role of $V$, and $\avqq$ in the role of $\avu$. 
Claim (b) follows from (a), since, clearly, $\fint{T}{T+\avq}{k}$ extends $\fint{T}{T}{k}$
\end{proof}


\subsection{Iterated approximations}

We will need to apply our technique of approximating theories to theories that 
themselves are already approximations\footnote{An example can be found in the proof of Lemma \ref{theorem:BroadHierarchyLemma}.}.  
To this end we generalise the definition of approximated theories to sequences of interpretations as follows.

\begin{definition}
Let $V^{[\la U, k \ra]} := V^{[U, k]}$. 
We recursively define  \[
    V^{[\la U_0, k_0 \ra, \dots, \la U_{n}, k_{n} \ra, \la U_{n + 1}, k_{n +1} \ra]}
\] for $n \geq 0$ to stand for ${\left( V^{[\la U_0, k_0 \ra, \dots, \la U_{n}, k_{n} \ra]} \right)} ^{ [U_{n + 1}, k_{n + 1} ] }$, i.e.:
\begin{align*}
    \axioms{  V^{[\la U_0, k_0 \ra, \dots, \la U_{n}, k_{n} \ra, \la U_{n + 1}, k_{n +1} \ra]}  }{x} \ & \eqbydef  \\
 (\exists\, p, \avu \bles x)\ \Big( & x = \ulcorner \avu \wedge (\dot{p} = \dot{p}) \urcorner \ \wedge \\
 & \axioms{  V^{[\la U_0, k_0 \ra, \dots, \la U_{n}, k_{n} \ra]} }{\avu} \ \wedge  \\ & \bewijs{U_{n + 1}}{p,\trans{\avu}{k_{n + 1} }} \Big).
\end{align*}
If $x$ denotes a finite sequence $\la U_0, k_0 \ra, \dots, \la U_n, k_n \ra$, then we understand $V^{[x,  \la U_{n + 1}, k_{n +1} \ra]}$ 
as $V^{[\la  U_0, k_0 \ra, \dots, \la U_n, k_n \ra,  \la U_{n + 1}, k_{n +1} \ra]}$.
\end{definition}

Theorem \ref{theorem:PLikePrincipleAtL}(a) can be adapted to this new setting, so that we get the following.

\begin{lemma}  \label{lemm:GeneralisationOftheorem:FinalPstyleTheorem}
Let $T$ be a  $ \Delta_1^{\sf b}$-axiomatised theory
and let $\alpha$ and $\beta$ be $T$-sentences. 
Let the variable  $x$ range over codes o sequences of pairs $\la U_i, k_i\ra$. We have: 
 \[
 \sonetwo\vdash (\forall x)(\forall k)\, ( k: (T + \avq) \rhd (T^{[x]} + \avqq)
        \to \Box_\sonetwo( (T + \avq) \rhd T^{[x, \la T + \avq, k \ra ]} + \avqq)).
        \]
\end{lemma}
\begin{proof}
       This is immediate from Theorem~\ref{theorem:PLikePrincipleAtL:general}, noting that the parameter $x$ does not affect the proof of that theorem.
\end{proof}

Again, it seems that there is no need to keep track of the formulas $\gamma$ in the \fint{T}{T+\gamma}{k} definition. 
Therefore, we shall, in the sequel,
simply work with sequences of interpretations of $T$ in $T$ rather than sequences of pairs of theory and interpretation.
The corresponding definition is as follows where $\la \ra$ denotes the empty sequence and for a sequence $x$, we use $x \star k$ or 
sometimes simply $x,k$ to denote the concatenation of $x$ with $\la k \ra$.
\begin{definition}
For $T$ a $\Delta_1^{\sf b}$-axiomatised theory we define $T^{[\la \ra]} := T$ and $T^{[x \star k]} := \left ( T^{[x]} \right )^{[T, k]}$.
\end{definition}
From now on, we shall write $T^{[k]}$ instead of $T^{[\la k\ra]}$. With the simplified notion of iteration we can formulate a friendlier \principle{P}-flavoured principle.

\begin{theorem}  \label{theorem:FinalPstyleTheorem}
Let $T$ be a  $ \Delta_1^{\sf b}$-axiomatised theory
 and let $\alpha$ and $\beta$ be $T$-sentences. 
Let $x$ range over sequences of interpretations. We have:
 \[
 \sonetwo\vdash  (\forall x)(\forall  k)\,\Bigl( (T + \alpha) \rhd (T^{[x]} + \beta)
        \to \Box_\sonetwo \,\bigl (k: ( T + \alpha) \rhd ( T^{[x, k ]} + \beta)\bigr )\Bigr ).
        \]
\end{theorem}

\section{A modal logic for approximation}

In this section, we will present a modal logical system to reason about interpretations and approximations based on them.

\subsection{The logic \atl}

We proceed to articulate modal 
principles reflecting facts about approximations. The main idea is to label our modalities with sequences $\svo$ of 
interpretation variables. Of course, in the arithmetical part, these sequences $\svo$ will indeed be interpreted via 
some map $\kappa$ as a sequence $\kappa (\svo)$ of translations from the language of $T$ to the language of $T$. 
In the next subsection we shall make the arithmetical reading precise but the idea is that $A \rhd^\svo B$ will stand for 
$T+ \alpha \rhd T^{[\kappa(\svo)]} + \beta$, whenever $A$ is interpreted by the $T$-sentence $\alpha$ and $B$ by $\beta$. 
Likewise, $\Box^\svo A$ will be interpreted as $\Box_{T^{[\kappa(\svo)]}}\alpha$. In the next section, we will see how we 
can avoid nonsensical interpretations $k$ so that the theories $T^k$ will always contain a minimum of arithmetic.

As in \cite{JoostenVisser:2004:Toolkit}, we will call our modal system \atl even though the system presented here slightly deviates from the one in \cite{JoostenVisser:2004:Toolkit}.
 We first specify the language. We have propositional variables
 $p_0,p_1,p_2,\ldots$. We will use $p,q,r,\ldots$ to range over them. Moreover, we have interpretation
 variables $k_0,k_1,k_2,\ldots$. We have one interpretation
 constant {\sf id}. 
 The meta-variables $k,\ell,m,\ldots$ will range over the interpretation
 terms (i.e.\ interpretation variables and {\sf id}). 
 The meta-variables $\svo, \svt, \svtr, \ldots$ will range over finite sequences of interpretation variables.\forlater{\albnote{Including the empty sequence?
 The clauses below suggest not.}
 \luknote{We need empty sequences for some of our current formulations at least. Wouldn't the clause be equivalent to saying this for the empty sequence case: If $A \rhd B$ is in the language, ... $A \rhd^k B$ is ...}}

 The modal language of \atl is the smallest language containing the propositional variables,
 closed under the propositional connectives, including $\top$ and $\bot$, and, given an interpretation term $k$, the modal operators $\Box^k$ and $\rhd^k$,
  and closed under
 the following rule.
 \begin{itemize}
 \item
 If $A \rhd^\svo B$ is in the language and $k$ is an interpretation term not contained in $\svo$, then  $A\rhd^{\svo, k}B$\/ is in the language. Similarly, for $\Box^{\svo, k} A$.
 \end{itemize}
We let $\Diamond^\svo A$ abbreviate $\neg\, \Box^\svo \neg\, A$. 
 We write $\rhd$ for $\rhd^{{\sf id}}$, and analogously for $\Box$ and $\Diamond$. The logic \atl~has axioms $\vdash A$ for any propositional logical tautology $A$ in the extended language. Moreover, \atl~has the obvious interchange rules to govern interaction between both sides of the turnstyle $\vdash$ based on the deduction theorem so that $\varDelta, \Gamma \vdash C \ \Leftrightarrow \varDelta \vdash \bigwedge \Gamma \to C $.
Apart from modus ponens, \atl~has the following axioms and rules. 
 
\[
\begin{array}{ll}
\finprin{L_1}{\svo} & \vdash \finmod{\Box}{\svo} (A \rightarrow B) 
\rightarrow (\finmod{\Box}{\svo} A \rightarrow \finmod{\Box}{\svo} B) \\
\finprin{L_2}{\svo, \svt} &  \vdash\finmod{\Box}{y} A \rightarrow \finmod{\Box}{\svo}
\finmod{\Box}{svt} A  \\
\finprin{L_3}{\svo} & \vdash \finmod{\Box}{\svo} 
(\finmod{\Box}{\svo} A \rightarrow A) \rightarrow 
\finmod{\Box}{\svo} A  \\\\
\finprin{J_1}{\svo} &  \vdash
\finmod{\Box}{\svo} (A \rightarrow B) \rightarrow A \finmod{\rhd}{\svo} B \\
\finprin{J_2}{\svo}{\sf a} &  \vdash (A \rhd B ) \wedge (B \finmod{\rhd}{\svo} C)
\rightarrow A \finmod{\rhd}{\svo} C \\
 \finprin{J_2}{\svo}{\sf b} &  \vdash (A \finmod{\rhd}{\svo} B ) \wedge \Box^{\svo}(B \to C)
\rightarrow A \finmod{\rhd}{\svo} C  \\
 %
\finprin{J_3}{\svo} &  \vdash (A \finmod{\rhd}{\svo} C) \wedge (B \finmod{\rhd}{\svo}C)
\rightarrow  {A\vee B}\, \finmod{\rhd}{\svo} C \\
\finprin{J_4}{\svo} &  \vdash
 A \finmod{\rhd}{\svo} B \rightarrow (\Diamond A \rightarrow \finmod{\Diamond}{\svo} B)  \\
%
%
\finprin{J_5}{\svo, \svt} &  \vdash
A\rhd^{\svo}\Diamond^{\svt}B \to A\rhd^{\svt}B \\\\
\finimpl{\svo,k} & \vdash \finmod{\Box}{\svo,k} A \rightarrow \finmod{\Box}{\svo} A \\
\finrhdimpl{\svo,k} &   \vdash A\finmod{\rhd}{\svo} B \rightarrow 
A \finmod{\rhd}{\svo,k} B  \\\\
\finprin{Nec}{\svo} & {\vdash A} \To {\vdash \Box^{\svo} A}\\
\finprin{P}{\svo, \svt, k} & 
{\Gamma,\varDelta,  \Box^\svt (A \rhd^{\svo, k} B)\vdash C} \To 
{\Gamma, A\rhd^\svo B \vdash C}  \\
\end{array}
\]

\noindent In the above, the rule  \finprin{P}{\svo, \svt, k} is subject to the following conditions:
\begin{enumerate}
\item $k$ is an interpretation variable;
\item $k$ does not occur in $\svo, \Gamma, A,B,C$;
\item $\varDelta$ consists of formulas of the form 
${E\rhd^{\svo, k} F} \to {E\rhd^\svo F}$ and\\ $\Box^\svo E\to \Box^{\svo, k} E$.
\end{enumerate} 

\subsection{Basic observations on \atl}

The first group of axioms $\finprin{L_1}{\svo}$-$\finprin{L_3}{\svo}$ express the straightforward generalisation of the regular provability axioms. 
The second group of axioms $\finprin{J_1}{\svo}$-$\finprin{J_5}{\svo, \svt}$ are 
the straightforward generalisations of the interpretability axioms. In particular, 
taking all interpretations to be the identity we retrieve all the regular axioms.

The third group of axioms tells us how we can vary the interpretability parameters. The Necessitation rule is as usual, and the $\finprin{P}{\svo, \svt, k}$ encodes the essential behaviour of approximations.

The following derivation shows how the $\finprin{P}{\svo, \svt, k}$ rules implies the axiom version CHANGE THIS WORD:
\[
\cfrac{\cfrac{\Box (A \rhd^{\svo, k} B)\vdash \Box (A \rhd^{\svo, k} B)}
{A\rhd^\svo B \vdash \Box (A \rhd^{\svo, k} B)}}{\vdash (A\rhd^\svo B) \to \Box (A \rhd^{\svo, k} B).}
\]

The use of a \principle{P}-flavoured rule instead of an axiom is suggested since it better allocates flexibility in collecting all applications of Lemma \ref{lemm:boxextensio} and Corollary \ref{cor:boxextensioProvability} in our reasoning. To be on the safe side, we consider that \atl~is presented using multi-sets so that we can allocate for applications of Lemma \ref{lemm:boxextensio} and Corollary \ref{cor:boxextensioProvability} \emph{after} a \finprin{P}{\svo, \svt, k}~rule is applied. Often we shall not mention all parameters of an axiom and, for example, just speak of the \finprin{P}{k}~rule instead of the \finprin{P}{\svo, \svt, k}~rule.

From Section \ref{setc:principles} onwards we shall put the logic \atl to work. Rather than giving formal proofs as a sequence of turnstyle statements we will describe such formal proofs. In doing so, we will call the licence to use 
$\Box^\svo E\to \Box^{\svo, k} E$ provided by \finprin{P}{\svo, \svt, k}:
\finboxext{k}, and we will call
    \footnote{A possible strengthening of \finprin{P}{\svo, \svt, k} is: \[
        {\Gamma,\varDelta,  \verz{\Box^\svt (A_i \rhd^{\svo,k} B_i \mid i<n+1)}\vdash C} \Rightarrow
            {\Gamma, \verz{A_i\rhd^\svo B_i\mid i<n+1} \vdash C}
    \]
    %
     putting  the obvious conditions on occurrences of $k$ and on $\varDelta$. We will not consider this strengthening in the paper.}
 the licence to use 
${E\rhd^{\svo, k} F} \to {E\rhd^\svo F}$:
\finintext{k}.

We observe that by taking the empty sequence we get various special cases of our axioms. For example, a special case of $\finimpl{x,k}$ would be $\vdash \finmod{\Box}{k} A \rightarrow \finmod{\Box}{} A$. Furthermore, successive applications of $\finimpl{\svo,k}$ yield $\vdash \finmod{\Box}{\svo} A \rightarrow \finmod{\Box}{} A$. Likewise, a special case of $\finrhdimpl{\svo,k}$ gives us $\vdash A\rhd B \rightarrow A \finmod{\rhd}{k} B$.


We observe that repeatedly applying $\finimpl{\svo,k}$ yields a generalisation of $\finprin{J_2}{\svo}{\sf b}$:~$(A \finmod{\rhd}{\svo} B ) \wedge \Box^{\svo\star y}(B \to C)
\rightarrow A \finmod{\rhd}{\svo} C$.

Furthermore, we observe that $\finprin{J_1}{\svo}$ follows from the classical  \principle{J_1} principle since
\[
\begin{array}{lll}
\finmod{\Box}{\svo}(A \rightarrow B) & \rightarrow & \Box (A \rightarrow B)\\
 & \rightarrow & A \rhd B\\
 & \rightarrow &  A \finmod{\rhd}{\svo} B.\\
\end{array}
\]

    We also observe that, if we drop the superscripts in \finprin{J_5}{\svo, \svt},
    we get the formula $A\rhd \Diamond B \ \to \ A\rhd B$ that is equivalent over \prin{J_1}, \prin{J_2} to the ordinary version $\Diamond A \rhd A$
    of \prin{J_5} as we saw in Lemma \ref{lemma:equivalentOfJ5}.

As a first and simple derivation in our system we have the following strengthening of the principle \principle{P_0}  in \atl~
    (recall that \principle{P_0} is the scheme $A \rhd \Diamond B \to \Box(A \rhd B)$).

    \begin{lemma}
    \label{gen_p0}
Let $\svo$ and $\svt$ be  arbitrary sequences of interpretations.
Let  $\varDelta$ consist of formulas of the form 
${E\rhd^{\svo, k} F} \to {E\rhd^{\svo}F}$ and $\Box^{\svo}E\to \Box^{\svo, k} E$ for some $k$ that does not occur in 
$\svo, \Gamma, A, B, C$.
We have the following rule to be derivable over \atl:
    \[
        {\Gamma,\varDelta,  \Box^\svt (A \rhd B)\vdash C} \To 
        {\Gamma, A\rhd^\svo \Diamond B \vdash C}
    \]
\end{lemma}
\begin{proof}
We assume ${\Gamma,\varDelta,  \Box^\svt (A \rhd B)\vdash C}$.
By \finprin{J_5}{\svo, k, {\sf id}} we know that \[\vdash A \rhd^{\svo, k} \Diamond B \to A\rhd B,\] so that we get  
${\Gamma,\varDelta,  \Box^\svt (A \rhd^{\svo, k} \Diamond B)\vdash C}$. An application of \finprin{P}{\svo, 
\svt, k} yields the required ${\Gamma, A\rhd^\svo \Diamond B \vdash C}$.
    \end{proof}

\section{Arithmetical semantics}

In order to set up arithmetical semantics, we would like to quantify over sensible translations. For example, a translation should at least map a minimum of arithmetic to provable sentences. However, how are we to separate the sensible from the non-sensical translations? In the first subsection we shall provide a construction to guarantee that we only use sensible translations. Then we shall define arithmetical semantics and prove a soundness theorem.

\subsection{A further modification of translations}\label{brutalesmurf}

To do interpretability logic we need that we have sufficient coding possibilities in each theory we consider. Suppose we already have a theory with coding and translation $k$
from the signature of $T$ to the signature of $T$. We want to insure that $T^{[T,k]}$ also has coding.  To do this we simply have to produce an improved version of
the modification trick we introduced in Section~\ref{voetbalsmurf}.

We fix our base theory $T$ of signature $\Theta$. Our coding will always be implemented via an interpretation of \sonetwo in $T$. We also fix such an interpretation, say $N$.
 Let $\alpha^\star$ be a conjunction of $T$-axioms that implies both $\mathfrak{id}_{\mathfrak A}$, where $\mathfrak A$ is the signature of arithmetic, and
 $(\sonetwo)^N$. We fix $\alpha^\star$ for the remainder of this section.

We can now specify our standard modification. Define, for  any translation $k$  of the signature of $T$ to the signature of $T$, the
disjunctive interpretation ${\sf n}(k)$ that is $k$ if $(\alpha^\star)^k$ and $ \mathsf{id}_\Theta$, otherwise.
Over predicate logic, we have, by a simple induction, that, for any $\Theta$-sentence $\varphi$,
\begin{equation}\label{equation:skCanBeSimplified}
\varphi^{\mathsf{n}(k)} \ \leftrightarrow \    \Big( (\alpha^\star)^k \wedge \varphi^{k}  \Big)
    \vee
    \Big( \neg (\alpha^\star)^k \wedge \varphi \Big).
\end{equation}
Since the needed induction to prove this is on the \emph{length} of $\varphi$ and since the proof can be uniformly constructed in p-time from $\varphi$, we have access to \eqref{equation:skCanBeSimplified} when reasoning inside $\sonetwo$.

We observe that, 
for example, in the formula $\exists k\, \Box_U \phi^{ \mathsf s(k) }$,
    the choice of whether $\mathsf s(k)$ will be equivalent to $\mathsf{id}_\Theta$ or to $k$
    will depend on whether $(\alpha^\star)^k$ holds under the $\Box_U$.
    In contrast, in the expression 
   $\exists k\, \Box_{U^{[\mathsf{s}(k)]}} \varphi$, the nature of $U^{[\mathsf{s}(k)]}$ depends on whether $(\alpha^\star)^k$ holds outside the box.
    
    Let us proceed by making some easy observations on $\mathsf{n}(k)$. In the following lemma, 
    we start by observing that regardless of the nature of $k$, the derived $\mathsf{s}(k)$ provides us 
    an interpretation of $\alpha^\star$ in $T$. Next, we see that any other interpretation of 
    $\alpha^\star$ in $T$ will also occur as an image of $\mathsf{s}$. Thus, modulo $T$-provable  equivalence, ${\sf n}(k)$ ranges precisely over
all interpretations of $\alpha^\star$.

\begin{lemma} \label{lemm:sk-props}
 We have, verifiably in \sonetwo, that, for all good translations $k$ and $j$,
 \begin{enumerate}
 \item
 $T\vdash (\alpha^\star)^{{\sf n}(k)}$,
 \item
 for any formula $\phi$ we have $T\vdash (\alpha^\star)^k \to (\phi^{{\sf n}(k)} \leftrightarrow \phi^k)$,
 \item 
 $T\vdash  \phi^{{\sf n}({\sf id}_\Theta)} \leftrightarrow \phi$\\ \textup(i.o.w., ${\sf n}({\sf id}_\Theta)$ is $T$-equivalent to ${\sf id}_\Theta$\textup),
 \item 
 $T\vdash  \phi^{{\sf n}(k\circ j)} \leftrightarrow \phi^{{\sf n}(k)\mathop{\circ} {\sf n}(j)}$\\
 \textup(i.o.w., {\sf n} $T$-provably commutes with composition of translations\textup).
 \end{enumerate}
\end{lemma}
\begin{proof}
    Let us prove the first claim.
    Reason in $\sonetwo$ and let $k$ be arbitrary.
    Now reason in $T$ or more formally, under the $\Box_T$.
    We distinguish cases.
    If $(\alpha^\star)^k$, then $\mathsf{s}(k) = k$,
        and $(\alpha^\star)^{k}$ holds by the case assumption.
    Otherwise, $\mathsf{s}(k) = \mathsf{id}_\Theta$.
    The choice of $T$ and $\alpha^\star$ (see beginning of the subsection)
        implies $T\vdash \alpha^\star$, as required.
        
    The second claim is immediate by the induction we already discussed.
    The third and fourth claims are easy.
\end{proof}
 %

We recall that where the lemma mentions the theory $T^{[ {T}, {\sf n}(k) ]}$ we really mean
    the theory axiomatised by 
    \begin{equation} \label{eqn:axTTAsk}
    \begin{aligned}
        \axioms{\fint{T}{{T}}{{\sf n}(k)}}{x}  = \ &
         ( \exists\, p, \varphi \bles x)\ 
        \Big(x= \ulcorner \varphi \wedge (\underline{p}=\underline{p})\urcorner\,
        \wedge \\ 
         & \axioms{T}{\varphi}\wedge \bewijs{{T}}{p,\trans{\varphi}{{\sf n}(k)}} \Big).
     \end{aligned}
    \end{equation}
In this formula we can expand $\trans{\varphi}{{\sf n}(k)}$ as in \eqref{equation:skCanBeSimplified}.

\begin{lemma}\label{gloglop}
$\sonetwo\vdash  \forall k\; \Box_{T^{[ T , {\sf n}(k) ]} }(\sonetwo)^N$. 
\end{lemma}

\begin{proof}
Reason in \sonetwo. 
Consider 
 any translation $k$ from the language of $T$ to the language of $T$. 
Lemma \ref{lemm:sk-props} tells us there is 
    a proof in $T$ of $(\alpha^\star)^{{\sf n}(k)}$. 
Hence, we have proofs $p_i$ of $(\alpha_i)^{{\sf s}(k)}$,
for (standardly) finitely many  $T$-axioms $\alpha_1, \dots, \alpha_n$.
We would like to show that $T^{[T, {\sf n}(k)]}$ 
    proves each of these $\alpha_i$, since then $T^{[T , {\sf n}(k)]}$ 
    proves $(\sonetwo)^N$.
We take arbitrary $\alpha_i$ and put 
$x = \ulcorner \alpha_i \wedge (\underline{p_i}=\underline{p_i})\urcorner$.
Clearly, this $x$ witnesses  \eqref{eqn:axTTAsk},
     the first two conjuncts of the body of \eqref{eqn:axTTAsk}.
Furthermore, $T$ proves $(\alpha_i)^{{\sf n}(k)}$ 
    because $p_i$ is a proof of this formula in $T$.
So, $ \alpha_i \wedge (\underline{p_i}=\underline{p_i})$ is an axiom
    of $T^{[T , {\sf n}(k)]}$, 
    whence $T^{[T , {\sf n}(k)]}$ 
    proves $\alpha_i$.
\end{proof}

\noindent
Recall that we work with the theories $T$ that interpret $\sonetwo$
    and that we fix a designated interpretation $N : T \rhd \sonetwo$.
We defined a variety of other theories of the form $T^{[x]}$, but we did not specify
    what interpretation of $\sonetwo$ we are supposed to bundle them with.
The preceding lemma tells us that we can reuse $N$.
Thus, 
    we will take $N$ as the designated interpretation of 
    \sonetwo in the $T^{[x]}$.

\subsection{Arithmetical soundness}
As before, we fix our base theory $T$ of signature $\Theta$, the interpretation $N$ of \sonetwo in $T$, the sentence $\alpha^\star$, and the mapping on translations {\sf n}.
Let us say that translations in the range of {\sf n} are \emph{good translation}. We call the formalised predicate of being good: {\sf good}.

As usual, the modal logics are related to arithmetic via \emph{realisations}. Realisations map the propositional 
variables to sentences in the language of arithmetic. However, we now also have to deal with the interpretation sequences. 
Thus, our realisations for the arithmetical interpretation are pairs $(\sigma, \kappa)$, where:
\begin{itemize}
    \item $\sigma$ maps the propositional variables to $T$-sentences, and
    \item $\kappa$ maps the interpretation variables to good translations  from the language of $T$ to the language of $T$.
    \end{itemize}
We stipulate that the $\sigma$ maps  all but finitely many arguments to $\top$  and, likewise,  that
    the $\kappa$ maps all but finitely many arguments to ${\sf n}({\sf id}_\Theta)$.
    The realisations are lifted to the arithmetical language in the obvious way by having them commute with the logical connectives and by taking:

\begin{align*}
    &(\Box^{k_1, \dots, k_n}A)^{\sigma,\kappa} := \Box_{T^{[ 
            \,
            \left\langle \nm{\kappa(k_1)} ,
            \dots,
         \nm{\kappa(k_n)} \right\rangle
            \,
        ]}}A^{\sigma,\kappa}, \mbox{ and }\\\\
    &(A\rhd^{k_1, \dots, k_n}B)^{\sigma,\kappa} :=\\ 
    & \hspace{4em} (T+A^{\sigma,\kappa}) \rhd (T^{[
        \,
        \left\langle  \nm{\kappa(k_1)},
        \dots,
         \nm{\kappa(k_n)} \right\rangle
        \,
    ]}+B^{\sigma,\kappa}).
\end{align*}
Here, in the context of $T$, we suppress the relativisation to $N$, it being the silent understanding that all coding is done inside $N$.
We observe that the nested modalities make sense because of Lemma~\ref{gloglop}. \emph{A central point here is that we allow
$\kappa$ to be an internal variable.} The transformation $T \mapsto T^{[\left\langle \nm{\kappa(k_1)},\dots  \nm{\kappa(k_n)} \right\rangle]}$ is, in essence,
a transformation of indices of theories and can, thus, be represented internally.

We note that the formula $(\Box^{k_1, \dots, k_n}A)^{\sigma,\kappa}$ will \emph{not} be generally equivalent to 
$((\neg \,A)\rhd^{k_1, \dots, k_n}\bot)^{\sigma,\kappa}$.

%

A modal formula $A$ will be arithmetically valid in $T$ and $N$, w.r.t. our choice of $\alpha^\star$, iff, for all $\sigma$, we have
$T\vdash \forall \kappa\, A^{\sigma,\kappa}$. We note that it is necessary that the quantifier over $\sigma$ is external, since the substitutions are
at the sentence level. However, the internal quantification over the $\kappa$ makes sense since these program transformations of of indices for theories.

\begin{theorem}[Arithmetical Soundness]
Let $T$ be a $\Delta_1^{\sf b}$-axiomatisable theory containing \sonetwo via $N$. We have, relative to a fixed $\alpha^\star$, 
\[
\Gamma \vdash_{\atl} A \ \ \To \ \  \text{for all } \sigma,\; \;T\vdash \forall \kappa\; \left (\,\bigwedge
\Gamma^{\sigma,\kappa}\to A^{\sigma,\kappa} \right ).
\]
\end{theorem}

\begin{proof}
We use induction on \atl proofs. 

The axiom \finimpl{\svo,k}: $\vdash \finmod{\Box}{\svo,k} A \rightarrow \finmod{\Box}{\svo} A$ is directly obtained by Corollary \ref{cor:boxextensioProvability}.\ref{tommel}.    
and \finrhdimpl{\svo,k} are immediate from Lemma~\ref{lemm:boxextensio}.
The principles \finprin{L_1}{\svo} to \finprin{J_4}{\svo} are simple. 


The principle \finprin{J_2}{\svo}{\sf b} is immediate from the observation that
$(T+\alpha) \rhd (T^{[\svo]}+\beta)$ and $\Box_{T^{[\svo]}}(\beta \to \gamma)$ imply $(T+\alpha) \rhd (T^{[\svo]}+\gamma)$.

The validity of \finprin{J_5}{\svo, \svt}
 follows
  from the observation that,
  \[\sonetwo\vdash \forall j,k\in {\sf good}\, (T^{[U, \nm{k}]}+{\sf con}(T^{[V, \nm j]}+B))\rhd
  (T^{[V,\nm j]}+B).\]
   This follows by the usual formalisation of Henkin's Theorem (see e.g., \cite{Visser:1991:FormalizationOfInterpretability, Joosten:2016:OreyHajek, visser:2018:InterpretationExistence}).

We now consider ${\sf P}^{\svo,\svt,k}$ which tells us that from ${\Gamma,\varDelta,  \Box^\svt (A \rhd^{\svo, k} B)\vdash C}$ we may derive 
${\Gamma, A\rhd^\svo B \vdash C}$. 

Consequently, for the closure of arithmetical validity under this rule we assume that 
\begin{equation}\label{equation:soundnessPAssumption}
\text{for all } \sigma, \ T \vdash (\forall \kappa) \  \Big( \bigwedge \Gamma^{\sigma, \kappa} \wedge \bigwedge \varDelta^{\sigma, \kappa} \wedge 
\big( \Box^{\svt} (A \rhd^{\svo,k}B)\big)^{\sigma, \kappa} \to C^{\sigma, \kappa} \Big)    
\end{equation}
and will need to prove
\begin{equation}\label{eq:PSoundToBeProven}
\text{for all } \sigma, \ T \vdash \forall \kappa \  \Big( \bigwedge \Gamma^{\sigma, \kappa} \wedge  \big(  A \rhd^{\svo}B\big)^{\sigma, \kappa} \to C^{\sigma, \kappa} \Big)    .
\end{equation}
To this end, we fix $\sigma$. We reason in $T$. We fix some $\kappa'$ and assume 
\begin{equation}\label{eq:PsoundAss}
\bigwedge \Gamma^{\sigma, \kappa'} \wedge  \big(  A \rhd^{\svo}B\big)^{\sigma, \kappa'} .
\end{equation}
We remind the reader that the modal interpretation variable $k$ is supposed to be fresh. Let $\kappa$ be as $\kappa'$ with the sole exception that 
\[
\kappa (k) \ : \ A^{\sigma, \kappa'} \rhd^{\kappa'(\svo)} B^{\sigma, \kappa'}  \mbox{ whence also } \kappa (k) \ : \ A^{\sigma, \kappa} \rhd^{\kappa(\svo)} B^{\sigma, \kappa}.
\]
Note that we can that the existence of a desired choice for $\kappa(k)$ is guaranteed by Assumption \eqref{eq:PsoundAss}.
By Lemma \ref{theorem:FinalPstyleTheorem} we get $\Box \Big( A^{\sigma, \kappa} \rhd^{\kappa(\svo,k)} B^{\sigma, \kappa} \Big)$. Moreover, by Lemma \ref{lemm:boxextensio}(\ref{tompoes}) and by Corollary \ref{cor:boxextensioProvability} we may conclude $\bigwedge \varDelta^{\sigma, \kappa}$ so that by \eqref{equation:soundnessPAssumption} we conclude $C^{\sigma, \kappa}$. Since $k$ does not occur in $C$ we may thus conclude $C^{\sigma, \kappa'}$ which finishes the proof of \eqref{eq:PSoundToBeProven} and hence the soundness of ${\sf P}^{\svo,\svt,k}$.
%
\ignore{We refer the reader to \cite{JoostenVisser:2004:Toolkit} for details. An important ingredient is given in Theorem \ref{theorem:FinalPstyleTheorem}.}
\end{proof}

\section{On principles in \ilal}\label{setc:principles}

In this section, we give arithmetical soundness proofs
for some well-known \inty principles that hold in all \rat. 
For this purpose we will employ the system \atl.

To avoid repeating too much content from \cite{JoostenVisser:2004:Toolkit},
    here we study only the following principles, 
    but with proofs written in more detail
    compared to \cite{JoostenVisser:2004:Toolkit}.
For other well-known principles we refer to \cite{JoostenVisser:2004:Toolkit}.
    \[
\begin{tabular}{ll}
\principle{W} & $\vdash A\rhd B\to A\rhd (B\wedge \Box\neg\,A)$ \\
\principle{M_0} & $\vdash A\rhd B \to (\Diamond A\wedge \Box C)\rhd (B\wedge\Box C)$ \\
\principle{R} &    $\vdash
A\rhd B \to \neg\, (A \rhd \neg\, C ) \rhd  B \wedge \Box C $
\end{tabular}
\]

\subsection{The principle \principle{W}}
\label{subsection:W}

We start with the \extil{P}-proof of the principle \principle{W},
    which we will later convert to an \atl{} proof of \principle{W}.

\begin{fact}\label{pw}
$\ilp \vdash \principle{W}$.
\end{fact}

\begin{proof}
We reason in \extil{P}. Suppose
$A\rhd B$. Then, $\Box (A \rhd B)$. Hence, $(*)$
$\Box (\Diamond A \rightarrow \Diamond B)$, and, thus, $(**)$
 $\Box (\Box \neg B \rightarrow \Box \neg A)$.

Moreover, from $A\rhd B$, we have 
$A \rhd (B \wedge \Box \neg A) \vee (B\wedge \Diamond A)$. 
So it is sufficient to show:
 $B \wedge \Diamond A \rhd B \wedge \Box \neg A$. 

We have:
\[
\begin{array}{llll}
B\wedge \Diamond A & \rhd &  \Diamond B & \mbox{  by $(*)$ }\\
 &\rhd & \Diamond (B \wedge \Box \neg B) & \mbox{ by \principle{L_3} } \\
 &\rhd & B \wedge \Box \neg B & \mbox{ by \prin{J_5} } \\
  &\rhd & B \wedge \Box \neg  A & \mbox{ by $(**)$}.
\end{array}
\]
\end{proof}

To prove arithmetical soundness of \principle{W} we will essentially replicate the modal proof of \principle{W} in \ilp. 
We first give a more formal version of the proof that uses the rule 
    \finprin{P}{x, y, k} in the way we formally defined it.
Afterwards we will give a more natural proof.    

\begin{lemma} \label{lemma-atl-w} The following holds:
    \begin{align*}
          &\Box(A \rhd^{[k]} B),\\
          &(B \wedge \Diamond A \rhd^{[k]} B \wedge \Box^{[k]} \neg B)
          \to 
          (B \wedge \Diamond A \rhd B \wedge \Box^{[k]} \neg B)\\
          \vdash_\atl{} &
          B \wedge \Diamond A \rhd B \wedge \Box \neg A.
    \end{align*}
\end{lemma}
\begin{proof}
    Reason in \atl. 
    Some simple uses of rules and axiom schemas of \atl{} are left implicit.
    \begin{align}
        \setcounter{equation}{0}
        \label{watl_asm1} &\Box(A \rhd^{[k]} B)   
          & \text{assump.}  \\
        \label{watl_asm2} & {\scriptstyle (B \wedge \Diamond A \rhd^{[k]} B \wedge \Box^{[k]} \neg B)
          \to 
          (B \wedge \Diamond A \rhd B \wedge \Box^{[k]} \neg B)}
          & \text{assump.}  \\
        \label{watl_a} &\Box(\Diamond A \to \Diamond^{[k]} B) 
          & \text{by }\eqref{watl_asm1}, \finprin{J_4}{k} \\
        \label{watl_b} &\Box(\Box^{[k]} \neg B \to \Box \neg A) 
          & \text{by }\eqref{watl_a} \\
        \label{watl_c} &\Diamond A \rhd \Diamond^{[k]} B
          & \text{by }\eqref{watl_a}, \prin{J_1} \\     
        \label{watl_d} &B \wedge \Diamond A \rhd \Diamond^{[k]} B
          & \text{by }\eqref{watl_c}, \prin{J_1}, \prin{J_2} \\
        \label{watl_e} &B \wedge \Diamond A \rhd \Diamond^{[k]} (B \wedge \Box^k \neg B)
          & \text{by }\eqref{watl_d}, \finprin{L_3}{k}, \prin{J_1}, \prin{J_2} \\
        \label{watl_f} &B \wedge \Diamond A \rhd^{[k]} B \wedge \Box^k \neg B
          & \text{by }\eqref{watl_e}, \finprin{J_5}{k} \\
        \label{watl_g} &B \wedge \Diamond A \rhd B \wedge \Box^{[k]} \neg B
          & \text{by }\eqref{watl_asm2}, \eqref{watl_f} \\
        \label{watl_h} &B \wedge \Diamond A \rhd B \wedge \Box \neg A
          & \text{by }\eqref{watl_b}, \eqref{watl_g} 
    \end{align}
\end{proof}

\begin{proposition} \label{prop-atl-w} The principle \principle{W} is arithmetically valid.
    \[
        \atl \vdash A \rhd B \to A \rhd B \wedge \Box \neg A.
    \]
\end{proposition}
\begin{proof}
    Reason in \atl{}.
    By \finprin{P}{k} and Lemma \ref{lemma-atl-w} we get \[
        A \rhd B \vdash_\atl B \wedge \Diamond A \rhd B \wedge \Box \neg A.\ (*)
    \]
    Now assume $A \rhd B$. 
    Combining $A \rhd B$ with $(*)$ we get
    \[
        B \wedge \Diamond A \rhd  B \wedge \Box \neg A.\ (**)
    \]
    Clearly $A \rhd B$ implies \[
        A \rhd (B \wedge \Box \neg A) \vee (B \wedge \Diamond A).\ (*{*}*)
    \]
    From $(**)$ and $(*{*}*)$ by \prin{J_3} we obtain $A \rhd B \wedge \Box \neg A$.
    Thus \[
        \atl \vdash A \rhd B \to A \rhd B \wedge \Box \neg A,
    \]
    as required.
\end{proof}

The proof presented in Proposition \ref{prop-atl-w} (and Lemma \ref{lemma-atl-w})
    resembles the proof we gave earlier demonstrating that $\extil{P} \vdash \principle{W}$.
However, the resemblance is not exactly obvious;
    we had to turn our proof ``inside-out'' in order to use the rule \finprin{P}{k} 
    (resulting in the contrived statement of Lemma \ref{lemma-atl-w}).
This can be avoided by applying the rule \finprin{P}{k} in a different way.

When we want to conclude something starting from $A \rhd^x B$,
    we introduce a fresh interpretation variable $k$
    and get\forlater{\luknote{Would ``and the new rule now allows us to infer'' be better than ``get''?}} $\Box^y (A \rhd^{x,k} B)$ (for whichever $y$ we find suitable).
Now we have to be a bit more careful; we can't end the proof before
    we eliminate this $k$. 
We also have to be careful in how we use the rules \finboxext{k} and \finintext{k}.
Essentially, any proof in the new form must be formalisable in the 
    system \atl{} as it was defined earlier.
Let us demonstrate this with the principle \principle{W}.
    
Reason in \atl. 
Suppose that $A\rhd B$. By \finprin{P}{k} we have  for some $k$ that
$ \Box (A \rhd^{[k]} B)$. Hence, by \finprin{J_4}{k}, we have  $(*)$
$\Box (\Diamond A \rightarrow \Diamond^{[k]} B)$ and, so,
$(**)$ $\Box (\Box^{[k]} \neg B \rightarrow \Box \neg A)$.

Moreover, from $A\rhd B$, we have $A\rhd (B\wedge \Box \neg A) \vee 
(B \wedge \Diamond A)$.
So it is sufficient to show $B\wedge \Diamond A\rhd B\wedge\Box\neg A$.
We have:
\[
\begin{array}{llll}
B\wedge \Diamond A & \rhd &  \Diamond^{[k]} B & \mbox{ by $(\ast)$}\\
 &\rhd & \Diamond^{[k]} (B \wedge \Box^{[k]} \neg B) & \mbox{ by \finprin{L_3}{k} }\\
 &\rhd & B \wedge \Box^{[k]} \neg B & \mbox{ by \finprin{J_5}{k} and \finintext{k} }\\
  &\rhd & B \wedge \Box \neg A & \mbox{ by $(**)$}.
\end{array}
\]

\subsection{The principle \principle{M_0}}

Another good test case is the principle \principle{M_0},
    since both $\extil{W} \nvdash \principle{M_0}$ and $\extil{M_0} \nvdash \principle{W}$.
Although we will later demonstrate the method for the principle \principle{R} too
    and $\extil{R} \vdash \principle{M_0}$,
    the proof for \principle{R} is more complex.
For this reason we include the principle \principle{M_0}.
    
We start with the \extil{P}-proof of \principle{M_0}: $A\rhd B \to (\Diamond A\wedge \Box C)\rhd (B\wedge\Box C)$.

\begin{fact}\label{pm0}
$\ilp \vdash \principle{M_0}$.
\end{fact}
\begin{proof} Reason in \extil{P}. 
\[
\begin{array}{llll}
A\rhd B & \to & \Box (A \rhd B ) &  \mbox{by \principle{P}} \\
&\to & \Box (\Diamond A \rightarrow \Diamond B) &  \mbox{by \prin{J_4}} \\
&\to & \Box (\Diamond A \wedge \Box C \rightarrow \Diamond B \wedge \Box C) & \\
&\to & \Diamond A \wedge \Box C \rhd \Diamond B \wedge \Box C &  \mbox{by \prin{J_1}} \\
&\to & \Diamond A \wedge \Box C \rhd \Diamond (B \wedge \Box C) &  \\
&\to & \Diamond A \wedge \Box C \rhd B \wedge \Box C & \mbox{by \prin{J_5}.}
\end{array}
\]
\end{proof}

Now we adapt this proof to fit \atl. 
We will not write the more formal version of the proof 
    (see the commentary in Subsection \ref{subsection:W}).

\paragraph{\principle{P}-style soundness proof of \principle{M_0}}

Reason in \atl. 

\[
\begin{array}{llll}
    A\rhd B & \to & \Box (A \rhd^{[k]} B ) &  \mbox{by \finprin{P}{k} } \\
    &\to & \Box (\Diamond A \rightarrow \Diamond^{[k]} B) &  \mbox{by \finprin{J_4}{k}} \\
    &\to & \Box (\Diamond A \wedge \Box C \rightarrow \Diamond^{[k]} B \wedge \Box C) & \\
    &\to & \Diamond A \wedge \Box C \rhd \Diamond^{[k]} B \wedge \Box C &  \mbox{by \prin{J_1}} \\
    &\to & \Diamond A \wedge \Box C \rhd \Diamond^{[k]} B \wedge \Box^{[k]} \Box C &  \mbox{by \finprin{L_2}{k}} \\
    &\to & \Diamond A \wedge \Box C \rhd \Diamond^{[k]} (B \wedge \Box C) &  \\
    &\to & \Diamond A \wedge \Box C \rhd^{[k]} B \wedge \Box C. & \mbox{by \finprin{J_5}{k}} \\
    &\to & \Diamond A \wedge \Box C \rhd B \wedge \Box C & \mbox{by \finintext{k}.}
\end{array}
\]

\subsection{The principle \principle{R}}

As a final example, we will prove that the principle \principle{R}: $A\rhd B \to \neg (A \rhd \neg C ) \rhd  B \wedge \Box C $ is arithmetically valid.

Before we see that $\ilp \vdash \principle{R}$, we first
prove an auxiliary lemma.

\begin{lemma}\label{lemm:kleinehulp}
$\il \vdash \neg (A \rhd \neg C )\wedge (A \rhd B )
\rightarrow \Diamond (B \wedge \Box C )$.
\end{lemma}

\begin{proof}
We prove the \il-equivalent formula
$(A \rhd B ) \wedge \Box (B \rightarrow \Diamond \neg C )
\rightarrow A \rhd \neg C $. But this is clear, 
as 
$\il \vdash (A \rhd B ) \wedge \Box (B \rightarrow \Diamond \neg C )
\rightarrow A \rhd  \Diamond \neg C$
and 
$\il \vdash  \Diamond \neg C \rhd \neg C $.
\end{proof}
 
\begin{fact}
$\ilp \vdash \principle{R}$.
\end{fact}

\begin{proof}
We reason in \extil{P}. Suppose
$A \rhd B$. It follows that   $\Box (A \rhd B )$. Using this we get:
\[
\begin{array}{llll}
\neg (A \rhd \neg C )& \rhd &\neg (A \rhd \neg C ) \wedge (A \rhd B )\\
 &\rhd& \Diamond (B \wedge \Box C ) & \mbox{by Lemma \ref{lemm:kleinehulp}}\\
 &\rhd& B \wedge \Box C  & \mbox{by \prin{J_5}.}
\end{array}
\]
\end{proof}

\paragraph{\principle{P}-style soundness proof of \principle{R}}
Reason in \atl. We first show that
$(A \rhd^{[k]} B ) \wedge \neg ( A \rhd \neg C )
\rightarrow \Diamond^{[k]} ( B \wedge \Box C )$.
We show an equivalent claim
$(A \rhd^{[k]} B ) \wedge \Box^{[k]} ( B \to \Diamond \neg C )
\rightarrow  A \rhd \neg C$.

Suppose that $A \rhd^{[k]} B$ and
$\Box^{[k]} (B \rightarrow \Diamond \neg C )$.
Thus, $A \rhd^{[k]} \Diamond \neg C$ by \finprin{J_2}{k}{\sf b}.
By  \finprin{J_5}{k} we get $A \rhd \neg C$, as required.
By necessitation,
\begin{equation}\label{equa:formallemmaatje}
\Box ((A \rhd^{[k]} B ) \wedge \neg ( A \rhd \neg C )
\rightarrow \Diamond^{[k]} ( B \wedge \Box C )).
\end{equation}
We now turn to the main proof. Suppose
$A \rhd B$. Then, for some $k$, we have $\Box (A \rhd^{[k]} B )$
and, thus,
\[
\begin{array}{llll}
\neg ( A \rhd \neg C ) &\rhd& 
\neg (A \rhd \neg C ) \wedge (A \rhd^{[k]} B )&
\\
 &\rhd&\Diamond^{[k]}(B \wedge \Box C)& \mbox{ by (\ref{equa:formallemmaatje})}\\
 &\rhd&B \wedge \Box C. & \mbox{ by \finprin{J_5}{k} and \finintext{k}}
\end{array}
\]

\section{Two series of principles}\label{section:TwoSeries}

In \cite{GorisJoosten:2020:TwoSeries} two series of interpretability principles are presented. One series is called the \emph{broad series}, denoted $\principle{R}^n$ (for $n\in \omega$). The other series is called the \emph{slim hierarchy}, denoted $\principle{R}_n$ (for $n\in \omega$). The latter is actually a hierarchy of principles of increasing logical strength.

Both series of principles are proven to be arithmetically sound in any reasonable arithmetical theory. The methods used to prove this soundness in \cite{GorisJoosten:2020:TwoSeries} involve definable cuts and in essence can be carried out in the system called ${\sf CuL}$. In the next two sections we will see how both series admit a soundness proof based on the method of finite approximations of target theories as embodied in our logic \atl. 
\ignore{We will also use this opportunity to state the results concerning modal semantics
    we obtained in collaboration with Jan Mas Rovira,
    which concern the two series.
The proofs of these results can be found in his Master's thesis
    \cite{MasRovira:2020:MastersThesis}.}

\subsection{Arithmetical soundness of the slim hierarchy}

As already mentioned, the slim hierarchy $\mathsf{R}_n$ defined 
    in \cite{GorisJoosten:2020:TwoSeries} 
    is actually a hierarchy.
Thus, to prove arithmetical soundness 
    it suffices to study a cofinal sub-series.
In our case we will study the certain sub-series $\widetilde{\mathsf R}_n$.
Let us define the original sequence first;
    even though we will use the sub-series for the most part.
Let $a_i, b_i, c_i$ and $e_i$ denote different propositional variables, for all $i\in \omega$.
We define a series of principles as follows.    
\[
\begin{array}{lll}
\principle{R_0} &:=& a_0 \rhd b_0 \to \neg (a_0 \rhd \neg c_0) \rhd b_0 \wedge \Box c_0\\
\ & \ & \ \\ 
\principle{R_{2n+1}}&:= & R_{2n} [\neg (a_n\rhd \neg c_n) /\neg (a_n\rhd \neg c_n) \wedge (e_{n+1}\rhd \Diamond a_{n+1});\\
\ & \ & 
b_n \wedge \Box c_n/b_n \wedge \Box c_n \wedge (e_{n+1} \rhd a_{n+1})]\\
\ & \ & \ \\ 
\principle{R_{2n+2}}&:=& R_{2n+1} [b_n/ b_n \wedge (a_{n+1}\rhd b_{n+1});\\
\ & \ & 
\Diamond a_{n+1} / \neg (a_{n+1}\rhd \neg c_{n+1});\\
\ & \ & 
(e_{n+1}\rhd a_{n+1})/ (e_{n+1}\rhd a_{n+1}) \wedge (e_{n+1} \rhd b_{n+1}\wedge \Box c_{n+1})
]\\
\end{array}
\]
    
We proceed with defining the sub-series $\widetilde{\mathsf R}_n$  (see \cite{GorisJoosten:2020:TwoSeries}, below Lemma 3.1) where the $\widetilde{\mathsf R}_n$ hierarchy exhausts the even entries of the original ${\mathsf R}_n$ hierarchy:    
\begin{align*}
    \mathsf{X}_0 &:= A_0 \rhd B_0 \\
    \mathsf{X}_{n + 1} &:= A_{n + 1} \rhd B_{n + 1} \wedge (\mathsf{X}_{n}) \\
    \mathsf{Y}_0 &:= \neg (A_0 \rhd \neg C_0) \\
    \mathsf{Y}_{n + 1} &:= \neg (A_{n + 1} \rhd \neg C_{n + 1}) \wedge (E_{n + 1} \rhd \mathsf{Y}_n) \\
    \mathsf{Z}_0 &:= B_0 \wedge \Box C_0 \\
    \mathsf{Z}_{n + 1} &:= B_{n + 1} \wedge (\mathsf{X}_{n}) \wedge \Box C_{n + 1}  \wedge (E_{n+1} \rhd A_n) \wedge (E_{n + 1} \rhd \mathsf{Z}_n) \\
    \widetilde{\mathsf R}_n &:= \mathsf{X}_n \to \mathsf{Y}_n \rhd \mathsf{Z}_n.
\end{align*}
For convenience, define $\mathsf{X}_{-1} = \top$.
With this we have $\mathsf{X}_n \equiv_\il A_n \rhd B_n \wedge (\mathsf{X}_{n - 1})$ for all $n \in \omega$. The first two schemas are:
\[
\begin{array}{lll}
\widetilde{\mathsf R}_0 &:=& A_0 \rhd B_0 \to \neg (A_0 \rhd \neg C_0) \rhd B_0 \wedge \Box C_0;\\

\widetilde{\mathsf R}_1 &:=& A_1 \rhd B_1 \wedge (A_0 \rhd B_0) \to \neg (A_1 \rhd \neg C_1) \wedge (E_1 \rhd \neg (A_0\rhd \neg C_0)) \ \rhd\\
\ & &\ \ \ \ \ \ \ \ \ \ \ \  \ \ B_1 \wedge (A_0 \rhd B_0) \wedge \Box  C_1  \wedge (E_1\rhd A_0) \wedge (E_1\rhd B_0 \wedge \Box C_0).\\

\end{array}
\]

In the proof that $\atl \vdash \widetilde{\mathsf R}_n$ (Theorem \ref{slim_soundness}) we use the following lemma.
\begin{lemma}
    \label{lemma_slim}
    For all $n \in \omega$, and all interpretation variables $k$: \[
        \atl \vdash (A_{n } \rhd^k B_{n } \wedge \bnb{\mathsf{X}_{n - 1})}
            \wedge \mathsf{Y}_{n } \to \Diamond^k\, \bnb{\mathsf{Z}_{n }}.
    \] 
\end{lemma}
\begin{proof}
    Let $n = 0$ and fix $k$. We are to prove \[
        \atl \vdash (A_0 \rhd^k B_0 \wedge \top) \wedge \neg(A_0 \rhd \neg C_0) \to \Diamond^k (B_0 \wedge \Box C_0).
    \]
    Equivalently, \[
        \atl \vdash (A_0 \rhd^k B_0) \wedge \Box^k(B_0 \to \Diamond \neg C_0) \to A_0 \rhd \neg C_0.
    \]
    Assume $(A_0 \rhd^k B_0) \wedge \Box^k(B_0 \to \Diamond \neg C_0)$.
    By $\finprin{J_2}{k}{\sf b}$, this yields $A_0 \rhd^k \Diamond \neg C_0$, whence by $\finprin{J_5}{k}$, $A_0 \rhd \neg C_0$.
    
    Let us now prove the claim for $n + 1$. Fix $k$.
    Unpacking, we are to show that:
    \begin{align*}
        \atl \vdash\, & (A_{n + 1} \rhd^k B_{n + 1} \wedge \bnb{\mathsf{X}_{n}}) \wedge 
        \neg (A_{n + 1} \rhd \neg C_{n + 1}) \wedge (E_{n + 1} \rhd \mathsf{Y}_n) \\
        & \to \Diamond^k\Big( B_{n + 1} \wedge \bnb{\mathsf{X}_{n}} \wedge \Box C_{n + 1}  \wedge (E_{n+1} \rhd A_n) \wedge (E_{n + 1} \rhd \mathsf{Z}_n)  \Big).
    \end{align*}
    Equivalently, we are to show that:
    \begin{align}
        \label{eq:goal}
        \begin{split}
            \atl \vdash\, & (A_{n + 1} \rhd^k B_{n + 1} \wedge \bnb{\mathsf{X}_{n}}) \wedge 
            (E_{n + 1} \rhd \mathsf{Y}_n) \\
            & \wedge \Box^k\Big( (B_{n + 1} \wedge \bnb{\mathsf{X}_{n}}) \to \Diamond \neg C_{n + 1} \vee \neg (E_{n+1} \rhd A_n) \vee \neg (E_{n + 1} \rhd \mathsf{Z}_n)\Big)\\
            &\to A_{n + 1} \rhd \neg C_{n + 1}.          
        \end{split}
    \end{align}
    
    Assume the conjunction on the left-hand side of \eqref{eq:goal}.
    The first and the third conjunct imply \[
        A_{n + 1} \rhd^k B_{n + 1} \wedge \bnb{\mathsf{X}_{n}} \wedge \Big( \Diamond \neg C_{n + 1} \vee \neg (E_{n+1} \rhd A_n) \vee \neg (E_{n + 1} \rhd \mathsf{Z}_n) \Big),
    \]
    whence by weakening,
    \begin{equation}
        \label{eq:with_long_disj}
        A_{n + 1} \rhd^k \bnb{\mathsf{X}_{n}} \wedge \Big( \Diamond \neg C_{n + 1} \vee \neg (E_{n+1} \rhd A_n) \vee \neg (E_{n + 1} \rhd \mathsf{Z}_n) \Big). 
    \end{equation}
    We now aim to get $A_{n + 1} \rhd^k \Diamond \neg C_{n + 1}$.
    To this end, we set out to eliminate the last two disjuncts within \eqref{eq:with_long_disj}.
    
    From $E_{n + 1} \rhd \mathsf{Y}_n$ (the second conjunct on the left-hand side of \eqref{eq:goal}) we have $E_{n + 1} \rhd \neg (A_n \rhd \neg C_n)$, thus $E_{n + 1} \rhd \Diamond A_n$, whence $\Box^k (E_{n + 1} \rhd A_n)$ by the generalised \principle{P_0} (Lemma \ref{gen_p0}). 
    We now combine $\Box^k (E_{n + 1} \rhd A_n)$ with \eqref{eq:with_long_disj}, simplify and weaken to obtain
    \begin{equation}
        \label{eq:with_shorter_disj}
        A_{n + 1} \rhd^k \bnb{\mathsf{X}_{n}} \wedge ( \Diamond \neg C_{n + 1} \vee \neg (E_{n + 1} \rhd \mathsf{Z}_n)).
    \end{equation}
    Thus, we have eliminated the second disjunct within \eqref{eq:with_long_disj}, and we are left to eliminate $\neg (E_{n + 1} \rhd \mathsf{Z}_n)$. 
    We will now use the second conjunct on the left-hand side of \eqref{eq:goal}, $E_{n + 1} \rhd \mathsf{Y}_n$, again.
    We wish to apply the rule $\finprin{P}{\la \ra, k, j}$, so assume $\Box^k(E_{n + 1} \rhd^j \mathsf{Y}_n)$.
    Combining $\Box^k(E_{n + 1} \rhd^j \mathsf{Y}_n)$ with \eqref{eq:with_shorter_disj} and unpacking $\mathsf{X}_n$, we obtain
    {\small
    \begin{equation}
        \label{eq:with_shorter_disj2}
        A_{n + 1} \rhd^k (A_{n } \rhd B_{n } \wedge \bnb{\mathsf{X}_{n - 1}}) \wedge (E_{n + 1} \rhd^j \mathsf{Y}_n) \wedge ( \Diamond \neg C_{n + 1} \vee \neg (E_{n + 1} \rhd \mathsf{Z}_n)).
    \end{equation}
    }
    Reason under $\Box^k$. 
    We wish to apply the rule $\finprin{P}{\la \ra, j, \ell}$ with $A_{n } \rhd B_{n } \wedge \bnb{\mathsf{X}_{n - 1}}$, so assume $\Box^j( A_{n } \rhd^\ell B_{n } \wedge \bnb{\mathsf{X}_{n - 1}} )$.
    Combining $\Box^j( A_{n } \rhd^\ell B_{n } \wedge \bnb{\mathsf{X}_{n - 1}} )$ 
    with $E_{n + 1} \rhd^j \mathsf{Y}_n$ we obtain (still under the $\Box^k$) that $E_{n + 1} \rhd^j ( A_{n } \rhd^\ell B_{n } \wedge \bnb{\mathsf{X}_{n - 1}} ) \wedge \mathsf{Y}_n $.
    Applying this to \eqref{eq:with_shorter_disj2} we may conclude
    \forlater{\luknote{Do you have a preference here, e.g. all same size, large only for some nesting level etc.?}
    \albnote{I am for a tasteful growth in size for the more outer brackets.}}
    {\small
    \[
         A_{n + 1} \rhd^k \Big( E_{n + 1} \rhd^j ( A_{n } \rhd^\ell B_{n } \wedge (\mathsf{X}_{n - 1}) ) \wedge \mathsf{Y}_n \Big) \wedge ( \Diamond \neg C_{n + 1} \vee \neg (E_{n + 1} \rhd \mathsf{Z}_n)).
    \]
    }
    The induction hypothesis allows us to replace $A_{n } \rhd^\ell B_{n } \wedge (\mathsf{X}_{n - 1} ) \wedge \mathsf{Y}_n$ with $\Diamond^\ell (\mathsf{Z}_n)$.
    \[
         A_{n + 1} \rhd^k (E_{n + 1} \rhd^j \Diamond^\ell (\mathsf{Z}_n)) \wedge ( \Diamond \neg C_{n + 1} \vee \neg (E_{n + 1} \rhd \mathsf{Z}_n)).
    \]
    By $\finprin{J_5}{j, \ell}$,
    \[
        A_{n + 1} \rhd^k (E_{n + 1} \rhd^\ell \mathsf{Z}_n) \wedge ( \Diamond \neg C_{n + 1} \vee \neg (E_{n + 1} \rhd \mathsf{Z}_n)).
    \]
    By our last application of \finprin{P}{\la \ra, j, \ell} and $\finintext{\ell}$, we can substitute $\rhd$ for $\rhd^\ell$:
    \[
        A_{n + 1} \rhd^k (E_{n + 1} \rhd \mathsf{Z}_n) \wedge ( \Diamond \neg C_{n + 1} \vee \neg (E_{n + 1} \rhd \mathsf{Z}_n)).
    \]
    Finally, we can simplify, weaken and apply \finprin{J_5}{k,\la \ra} to obtain $A_{n + 1} \rhd \neg C_{n + 1}$.
\end{proof}
We can now prove soundness for the slim hierarchy. 
It suffices to do this for the cofinal sub-hierarchy $\widetilde{\mathsf R}_n$.
\begin{theorem}
    \label{slim_soundness}
    For all $n \in \omega$, $\atl \vdash \widetilde{\mathsf R}_n$.
\end{theorem}
\begin{proof}
    Let $n \in \omega$ be arbitrary. 
    Assume $\Box(A_n \rhd^k B_n \wedge (\mathsf{X}_{n - 1}))$.
    Clearly \[ 
        \mathsf{Y}_n \rhd (A_n \rhd^k B_n \wedge \bnb{\mathsf{X}_{n - 1}}) \wedge \mathsf{Y}_n.
    \]
    Now Lemma \ref{lemma_slim} implies
    \[
        \mathsf{Y}_n \rhd \Diamond^k\bnb{\mathsf{Z}_n},
    \]
    whence by \finprin{J_5}{\la \ra,k},
    \[
        \mathsf{Y}_n \rhd^k \mathsf{Z}_n.
    \]
    By the rule $\finprin{P}{k}$, we can replace our assumption $\Box^k(A_n \rhd B_n \wedge \bnb{\mathsf{X}_{n - 1}})$ with $\mathsf X_n$.
    Furthermore, by the same application of $\finprin{P}{k}$, and by \finintext{k}, we have $\mathsf{Y}_n \rhd \mathsf{Z}_n$.
    Thus, $X_n \to \mathsf{Y}_n \rhd \mathsf{Z}_n$, i.e.\ $\widetilde{\mathsf R}_n$.
\end{proof}

\ignore{Finally, as we announced earlier, we quote the result obtained in collaboration 
    with Jan Mas Rovira.
To state the generalised frame condition for the principle \principle{R_1}
    (which lies strictly between $\widetilde{\mathsf R}_0$ and $\widetilde{\mathsf R}_1$) 
    we let \(R^{-1}[E] := \{x : (\exists  y\in E) xRy\}\), and \(R_x^{-1}[E] := R^{-1}[E]\cap R[x]\). 

\begin{theorem}
    The frame condition for the principle \principle{R_1} with respect
    to generalised Veltman semantics is the following condition:
\begin{flalign*}
\forall w,x,u,\mathbb{B},\mathbb{C},\mathbb{E} & \bigg(wRxRuS_w\mathbb{B}, \mathbb{C}\in \mathcal{C}(x,u) \\
& \ \Rightarrow \  (\exists  \mathbb{B}'\subseteq \mathbb{B})\Big(xS_w\mathbb{B}',R[\mathbb{B}']\subseteq \mathbb{C},(\forall v\in \mathbb{B}')(\forall c\in \mathbb{C}) \\
& \hspace{2.5em} (vRcS_xR_x^{-1}[\mathbb{E}]\Rightarrow (\exists  \mathbb{E}'\subseteq \mathbb{E})cS_v\mathbb{E}')\Big)\bigg).
\end{flalign*}
\end{theorem}
\begin{proof}
    Please see \cite{MasRovira:2020:MastersThesis} for the proof
    (including a formalisation in Agda).
\end{proof}
}

\subsection{Arithmetical soundness of the broad series}

\newcommand{\uu}{{\sf U}}

In order to define the second series we first define a series of auxiliary formulas. For any $n\geq 1$ we define the schemata $\uu_n$ as follows.
\forlater{\luknote{Old version, not sure why we wrote it like that: $\uu_{n+2} := \Diamond((D_{n+1}\rhd D_{n+2})\wedge\uu_{n+1})$}}
\begin{align*}
\uu_1 &:= \Diamond\neg(D_1\rhd \neg C),\\
\uu_{n+1} &:= \Diamond((D_{n}\rhd D_{n+1})\wedge\uu_{n}).
\end{align*}
Now, for $n\geq 0$ we define the schemata $\principle R^n$ as follows.
\begin{align*}
\principle{R}^0 &:= A\rhd B\rightarrow\neg(A\rhd \neg C)\rhd B\wedge\Box C,\\
\principle{R}^{n+1} &: = A\rhd B\rightarrow\Big( \uu_{n+1}\wedge(D_{n+1}\rhd A)\Big)\rhd B\wedge\Box C.
\end{align*}

As an illustration we present the first three principles.
\[
\begin{array}{lll}
\principle{R}^0 & :=  & A \rhd B \to \neg (A \rhd \neg C) \rhd B \wedge \Box C;\\
\principle{R}^1 &:=& A \rhd B \to \Diamond \neg(D_1 \rhd \neg C) \wedge (D_1 \rhd A)  \rhd B \wedge \Box  C;\\
\principle{R}^2 &:=& A \rhd B \to  \Diamond\Big[ (D_1 \rhd  D_2) \wedge\Diamond\neg(D_1 \rhd \neg C)\Big] \wedge (D_2 \rhd A) \rhd B \wedge \Box  C.\\


%
%
\end{array}
\]

When working with this series it is convenient to also have the following schemas:
\begin{align*}
    \mathsf V_1 &:= \Box (D_1 \rhd \neg C), \\
    \mathsf V_{n + 1} &:= \Box(D_{n} \rhd D_{n + 1} \to V_n) \mbox{ for $n\geq 1$}.
\end{align*}
Alternatively, we could have defined $\mathsf V_n := \neg \mathsf U_n$ for $n\geq 1$. 
\begin{lemma}\label{theorem:BroadHierarchyLemma}
    \label{lemma_broad1}
    For all $n \in \omega \setminus \{0\} $, and all finite sequences $\svo$ consisting of interpretation variables: \[
        \atl \vdash D_n \rhd^\svo \Diamond \neg C \to \mathsf V_{n}.
    \] 
\end{lemma}
\begin{proof}
    Let $n = 1$ and $\svo$ be arbitrary. We want to prove that $\atl \vdash D_1 \rhd^\svo \Diamond \neg C \to \Box (D_1 \rhd \neg C)$. 
    This is an instance of the generalised \principle{P_0} schema as we stated in Lemma \ref{gen_p0}.
    
    Let us now prove the claim for $n + 1$. Thus, we fix an arbitrary sequence of interpretations $\svo$. We are to show that 
    \[
        \atl \vdash D_{n + 1} \rhd^\svo \Diamond \neg C \to \Box(D_{n} \rhd D_{n + 1} \to \mathsf V_n).
    \] 
    Thus, reasoning in \atl, we assume $D_{n + 1} \rhd^\svo \Diamond \neg C$.
    We now wish to apply the rule $\finprin{P}{k}$ with this formula, where $k$ is an arbitrary variable not used in its left or right side or $\svo$.
    So, assume $\Box(D_{n + 1} \rhd^{\svo, 
 k   } \Diamond \neg C)$. 
    Reason under a box. 
    Assume $D_n \rhd D_{n + 1}$. 
    Now $D_n \rhd D_{n + 1}$ and $D_{n + 1} \rhd^{\svo, 
    k} \Diamond \neg C$ imply $D_n \rhd^{\svo, 
    k} \Diamond \neg C$.
    By the necessitated induction hypothesis, this implies $\mathsf V_n$. Thus, we find $\Box(D_n \rhd D_{n + 1} \to \mathsf V_n)$, as required.
\end{proof}

\begin{lemma}
    For all interpretation variables $k$ we have the following:
    \label{lemma_broad2}
    \[
        \atl \vdash \mathsf U_{ n } \wedge (D_{ n } \rhd A) \wedge (A \rhd^k B) \rhd^k B \wedge \Box C.
    \]
\end{lemma}
\begin{proof}
    It is clear that the claim to be proved follows by necessitation, $\prin{J_1}$, and $\finprin{J_5}{\la \ra,k}$ from the following:
    \[
        \atl \vdash \mathsf U_{ n } \wedge (D_{ n } \rhd A) \wedge (A \rhd^k B) \to \Diamond^k (B \wedge \Box C).
    \]
    This formula is equivalent to \[
        (D_{ n } \rhd A) \wedge (A \rhd^k B) \wedge \Box^k (B \to \Diamond \neg C) \to \mathsf V_{ n }.
    \]
    Assuming the left-hand side, we get $D_n \rhd^k \Diamond \neg C$, whence $V_n$ by Lemma \ref{lemma_broad1}.
\end{proof}

\begin{theorem}
    \label{soundness_broad}
    For all $n \in \omega$, $\atl \vdash \mathsf R^n$.
\end{theorem}
\begin{proof}
    Case $n = 0$ is clear.
    Let $n > 0$ be arbitrary and let us prove $\mathsf R^n$. 
    Reason in \atl.
    Assume $A \rhd B$. 
    We wish to apply the rule $\finprin{P}{k}$ here.
    So, assume $\Box(A \rhd^k B)$. We have: \[ 
        \mathsf U_{ n } \wedge (D_{ n } \rhd A) \rhd
        \mathsf U_{ n } \wedge (D_{ n } \rhd A) \wedge (A \rhd^k B).
    \]
    Lemma \ref{lemma_broad2} and the rule \prin{J_2} imply \[
        \mathsf U_{ n } \wedge (D_{ n } \rhd A) \rhd^k B \wedge \Box C,
    \]
    and by $\finintext{k}$,
    \[
        \mathsf U_{ n } \wedge (D_{ n } \rhd A) \rhd B \wedge \Box C.
    \]
    So we are done.
\end{proof}

\bibliographystyle{alpha}
\bibliography{levref}

\end{document}